\documentclass[12pt,reqno]{amsart}

\usepackage[latin1]{inputenc}
\usepackage{amssymb,amsmath,amstext,amsthm,amsfonts}
\usepackage{amsfonts}
\usepackage{amsmath}
\usepackage{amssymb}
\usepackage{amsthm}
\usepackage{pstricks,pstricks-add,pst-math,pst-xkey}

\newcommand{\R}{\mathbb{R}}

\newcommand{\bdm}{\begin{displaymath}}
\newcommand{\edm}{\end{displaymath}}
\theoremstyle{definition}

\newtheorem{lem}{Lemma}
\newtheorem{thm}{Theorem}
\newtheorem{defn}{Definition}
\newtheorem{prop}{Proposition}
\newtheorem{rem}{Remark}
\newtheorem{eg}{Example}

\title[Metrics on tiling spaces]{Metrics on tiling spaces, local isomorphism and an application of Brown's Lemma.}

\author{Rui Pacheco}
\address{Universidade da Beira Interior\\
Rua Marquês d'Ávila e Bolama, 6200-001 Covilhã, Portugal}
\email{rpacheco@ubi.pt} \urladdr{http://www.mat.ubi.pt/$\sim$rpacheco}

\author{Helder Vilarinho}
\address{Universidade da Beira Interior\\
Rua Marquês d'Ávila e Bolama, 6200-001 Covilhã, Portugal}
\email{helder@ubi.pt} \urladdr{http://www.mat.ubi.pt/$\sim$helder}

\date{\today}

\thanks{The authors were partially supported by the Portuguese Government through FCT, under the project PEst-OE/MAT/UI0212/2011 (CMUBI)}

\begin{document}
\begin{abstract}We give an application of a topological dynamics version of multidimensional Brown's lemma to tiling theory: given a tiling  of an Euclidean space and a finite geometric pattern  of points $F$,  one can find a patch such that, for each scale factor $\lambda$,  there is a vector $\vec{t}_\lambda$ so that copies of this patch appear in the tilling ``nearly" centered on $\lambda F+\vec{t}_\lambda$ once we allow ``bounded perturbations" in the structure of the homothetic copies of $F$. Furthermore,  we introduce a new unifying setting for the study of tiling spaces which allows rather general group ``actions" on patches and we discuss the local isomorphism property of tilings within this setting.\end{abstract}

\maketitle

\section{Introduction}

The main idea of Ramsey theory  is that arbitrarily large
sets cannot avoid a certain degree of ``regularity". This
is exemplarily illustrated by Gallai's theorem, a
multidimensional version of the seminal van der Waerden's
theorem, which asserts that, given a finite coloring of
$\mathbb{Z}^n$, any finite subset $F$ of $\mathbb{Z}^n$ has
a monochromatic homothetic copy $\lambda F+\vec{t}$. As
shown recently  by de la Llave and Windsor  \cite{LW}, this
result has an interesting consequence in tiling theory.
Roughly speaking, given a tiling $y$ of $\mathbb{R}^n$ and
a finite geometric pattern $F\subset \mathbb{R}^n$ of
points, one can find a patch $y'$ of $y$ so that copies of
$y'$ appear in $y$ ``nearly" centered on some homothetic
version of the pattern. Hence, even if some sets of tiles
tile the plane only non-periodically (perhaps the Penrose
tiles  \cite{GS,Pen} are the most famous sets of tiles in
such conditions),  any tiling must exhibit some kind of
``approximate periodicity". The proof uses Furstenberg's
topological multiple recurrence theorem  for commuting
homeomorphisms \cite{Fu} (which is a topological dynamic
version of Gallai's theorem) applied to certain tiling
spaces $Y$ equipped with suitable metrics $d$.

Starting with a finite set $\mathcal{F}$ of \emph{prototiles}, three distinguished cases were considered by de la Llave and Windsor: the tiles of $y\in Y$ are obtained by taking translated (resp. direct isometric) copies of the prototiles, each $y\in Y$ exhibits \emph{finite local complexity}, which is a property that, roughly speaking, does not allow, for example, two  tiles to slide along their common boundary, and the distance $d$ makes two tilings close if they agree in a large ball about the origin up to a small translation (resp. direct isometry); thirdly, the tiles of $y\in Y$ are obtained by taking  direct isometric copies of the prototiles,  and  $d$ makes two tilings close if they agree in a large ball about the origin up to  small direct isometries (rigid motions) of each individual tile. In all these three cases, the metric spaces $(Y,d)$ are compact and the group of translations acts continuously on it. As a matter of fact, these are the standard three metrics that are frequently applied to tilings in the literature. In the first two cases, it is now well known that the corresponding topological tiling spaces can be seen as inverse limits of simpler topological spaces (see \cite{Sad} and references therein). More recently, Frank and Sadun \cite{FS} have shown that, with respect to the third metric,  some spaces without finite local complexity can also be understood as inverse limits.

In the present paper,  we start by establishing, in Section \ref{I}, a new unifying setting for the study of tiling spaces. The metrics we use make two tilings close if they agree in a large ball about the origin up to  small perturbations, which  are understood as elements of certain right invariant metric groups acting on patches.  These group ``actions" are not necessarily associated to
a  subgroup of isometries or even to a group of continuous transformations on the Euclidean space $\mathbb{R}^n$.  Hence, we can think on tilings formed by an infinite (up to isometries) variety  of tile types. Very recently, Frank and Sadun \cite{FS2} have studied certain spaces of tilings -- the \emph{fusion tiling spaces with infinite local complexity} -- exhibiting also an infinite variety of tile types. In such spaces they considered the topology  induced by a metric that takes into account the Hausdorff distance between the \emph{skeleton}  of tilings and a certain distance on the space of tile labels (the skeleton of a tiling of $\mathbb{R}^n$ is  the union of all the boundaries of its tiles). However this metric  does not carry any geometric information with respect to the underlying groups.
Our setting not only allows us
 to avoid to treat separately the three standard metrics, but also it is suitable to deal with more general group ``actions" on patches. Moreover, with minor technical changes we could also associate a set of labels to each tile and think on groups inducing actions on these sets of labels, giving in this way an alternative approach to handle with the fusion tiling spaces with local infinite complexity studied by Frank and Sadun \cite{FS2}. Here we will not explore the inverse limit point of view. The reader who are just interested in the usual metric spaces $(Y,d)$ of tilings exhibiting finite local complexity under isometries can skip this entire section (just taking into account Examples \ref{exampleG}, \ref{ex3} and \ref{ex4} for notational conventions), since the results for this particular case are well known and can be easily found in the literature, namely: $d$ defined by \eqref{distance} provides $Y$ with a compact metric structure, with respect to which any isometry acts continuously.

A not so famous Ramsey-type result is the so called Brown's lemma \cite{B1,B2}. Observe that Gallai's theorem does not say nothing, apart its existence, about the scale factor $\lambda$. On the other hand, the multidimensional version of Brown's lemma asserts  that one can take any $\lambda$ once we allow ``bounded perturbations" in the structure of the homothetic copies of $F$. In this paper,  we give an application of a topological dynamics version (Lemma \ref{BrownTop}, Section \ref{bltdv}) of this result  to tiling theory (Theorem \ref{BT}, Section \ref{btiling}).

This application  states that, given a tiling  of an Euclidean space and a finite geometric pattern  of points $F$,  one can find a patch such that, for each scale factor $\lambda$,  there is a vector $\vec{t}_\lambda$ so that copies (with respect to the underlying right invariant metric group) of this patch appear in the tilling ``nearly" centered on $\lambda F+\vec{t}_\lambda$ once we allow ``bounded perturbations" in the structure of the homothetic copies of $F$. For a tiling $y$ of $\mathbb{R}^n$ with finite complexity under isometries, and considering for instance the metric $d_\mathcal{I}$  on the group of direct isometries of  $\mathbb{R}^n$ defined by \eqref{di}, this reads as follows:
\smallskip
\begin{quote}
    \emph{Given  $\epsilon>0$ and a finite subset $F=\{\vec{v}_1,\ldots, \vec{v}_l\}$ of $\mathbb{R}^n$, there exist $q\in\mathbb{N}$ and a finite patch $y'$, whose support   contains the ball of radius $1/\epsilon$ about the origin,   satisfying:
for  each $\lambda>0$ and $\vec{v}_i\in F$, there exist a vector
$\vec{t}_\lambda\in \mathbb{R}^n$ (not depending on $i$) and an isometry $ g_{\lambda,i}$, with
$d_\mathcal{I}({g}_{\lambda,i}, Id)<\epsilon$, such that
   the translated copy of $g_{\lambda,i}(y')$  by $\lambda\vec{v}_i+\vec t_\lambda+\vec w_{\lambda,i}$ is a patch of $y$ for some $\vec{w}_{\lambda,i}= \sum_{j=1}^l\alpha_{j}\vec{v}_j$, with $|\alpha_{j}|\leq q$ for all $j$ (moreover, $\alpha_{j}\in\mathbb{Z}$ if $j\neq i$).}
\end{quote}
\smallskip

Recall that, in a topological dynamical system, a point $x$ is \emph{almost periodic} if, for every neighborhood $U$ of $x$, the set  of ``return times" to $U$  is \emph{relatively dense}. For tiling spaces, almost periodicity implies  the \emph{local isomorphism} property, which is
 a property that a tiling of an Euclidean space might have which  expresses a certain ``regularity". This is not an unusual property. For example, all Penrose tilings satisfy it \cite{GS}.  Radin and Wolff \cite{Rad} proved that any compact tiling space $Y$ with finite local complexity under direct isometries must admit a tiling satisfying  the local isomorphism property. We emphasize that de la Llave and Windsor's results and Theorem \ref{BT} in the present paper still hold
 even  in the lack of almost periodicity.

\section{Metrics on tiling spaces}\label{I}

Roughly speaking, a tiling of $\mathbb{R}^n$ is an arrangement of tiles that covers $\mathbb{R}^n$ without overlapping. Typically  one starts with a fixed finite set $\mathcal{F}$ of ``prototiles" and each tile is an isometric copy of some prototile \cite{GS,LW,Rad,Rob,Sad}. Denote by $Y_\mathcal{F}$ the set of all tilings of $\mathbb{R}^n$  obtained  in this way from a given prototile set $\mathcal{F}$. There is a natural metric on $Y_\mathcal{F}$, $d_\mathcal{F}$, with respect to which two tilings are close if their skeletons are close (with respect to the usual Hausdorff distance between compact sets of $\mathbb{R}^n$) on a large ball about the origin \cite{Rad}. However, it is possible to provide $Y_\mathcal{F}$ with alternative (but equivalent under certain conditions) metrics which carry much more geometric information. A familiar way to do this is by making two tilings close if they agree in large ball about the origin up to a small translation \cite{LW,Rob,Sad}.  Instead of translations, one could consider in this definition any subgroup of the group $\mathcal{I}$ of rigid motions group or even consider \emph{piecewise} rigid motions, that is, rigid motions of each individual tile \cite{LW,Sad}. In this section, we pretend to establish a new setting which unifies and generalizes the previous cases. In order to get some motivation in mind for our formalism, start by thinking in the case of piecewise rigid motions. Given a patch $x'$ with a number $n$ of tiles, the admissible perturbations on this patch is given by some elements of the group $\mathcal{I}^n$. Of course, since we have to avoid overlapping of tiles, not all elements of $\mathcal{I}^n$ induce admissible perturbations on $x'$. Hence we do not have a proper action of $\mathcal{I}^n$ on patches. Moreover, the admissible perturbations and its size are induced by elements of different metric  groups, according to the number of tiles of the patch. Perturbations on the subpatches of $x'$ must be compatible with  perturbations on $x'$. From this example we see that a general setting allowing a considerable variety of group ``actions" on patches must be to some extent similar to a sheaf construction.

\vspace{.20in}

Consider $\mathbb{R}^n$ with its usual Euclidean norm $\|\cdot\|$ and write $B_r=\{\vec v\in\mathbb{R}^n:\,\|\vec v\|\leq r\}$. A set $D\subset \mathbb{R}^n$ is called a \emph{tile} if it is compact, connected and equal to the closure of its interior. A \emph{tiling} of a subset $S\subseteq \mathbb{R}^n$ is a collection  $x=\{D_i\}_{i\in I}$ of tiles such that: \begin{itemize}
    \item[(T$_1$)] $S=\bigcup_{i\in I}D_i$;
    \item[(T$_2$)] $D^{^\circ}_i\cap D^{^\circ}_j=\emptyset$, for all $i,j\in I$ with $i\neq j$.
  \end{itemize}
  In this case we say that $S$ is the \emph{support} of $x$ and write $S=\mathrm{supp}(x)$. We denote by $X(S)$ the set of all tilings of $S$ and define $\mathcal{X}:=\bigcup_S X(S).$ If  $x,x'\in\mathcal{X}$ and $x'\subseteq x$, then $x'$ is called a \emph{patch} of $x$.

The set $\mathcal{X}$ of all tilings is too large. Typically one focuses attention on certain subsets of $\mathcal{X}$, such as in the case of tiling spaces with finite local complexity under rigid motions. So, let $\mathcal{Y}$ be a subset of $\mathcal{X}$ satisfying:
  \begin{itemize}
    \item[($\mathcal{Y}_1$)] for all $x\in \mathcal{Y}$ and $x'\subseteq x$, we have $x'\in\mathcal{Y}$;
    \item[($\mathcal{Y}_2$)] given $x\in\mathcal{X}$, if $x'\in \mathcal{Y}$ for all  $x'\subsetneq x$ with bounded support, then   $x\in \mathcal{Y}$.
  \end{itemize}
Axiom ($\mathcal{Y}_1$) will allow us to go from constructions on tilings to constructions on their patches, whereas  ($\mathcal{Y}_2$) essentially states the converse, that is, it will allow us to go from local constructions on patches  to global constructions on tilings.
If $K\subset \mathbb{R}^n$ is compact, we denote by $x[[K]]$ the set of all patches $x'$ of $x\in \mathcal{Y}$ with bounded support satisfying $K\subseteq \mathrm{supp}( x')$. Clearly, if $x'\in x[[K']]$ and $x''\in x[[K'']]$ then $x'\cap x''\in x[[K'\cap K'']]$.

Now, fix an equivalence relation on $\mathcal{Y}$ and denote by $[x]$ the equivalence class of $x\in\mathcal{Y}$. Associate to each equivalence class $[x]$ a metric group $(G{[x]},d_{G{[x]}})$ and to each ordered pair $(y,z)$ of patches in $[x]$ a non-empty subset $\gamma(y,z)$ of $G[x]$
such that:
\begin{itemize}
  \item[($\Gamma_1$)] $\gamma(y,z)=\gamma(z,y)^{-1}$;
  \item[($\Gamma_2$)] $\gamma(z,w)\gamma(y,z)=\gamma(y,w)$;
  \item[($\Gamma_3$)] $\gamma(x,y)\cap\gamma(x,z)=\emptyset$ if $y\neq z$.  \end{itemize}
If $g\in\gamma(x,y)$, we write $y=g(x)$ and say that $g$\emph{ transforms} $x$ in $y$. Observe that $I\!d_{G[x]}\in \gamma(y,y)$ for all $y\in[x]$. Define $$\gamma(x,[x]):=\bigcup_{y\in[x]} \gamma(x,y)\subseteq G[x].$$
For example, in the case of piecewise rigid mitions, we consider two tilings $x$ and $y$ equivalent if the tiles of $x$  can be individually perturbed by rigid motions in order to obtain $y$. If $x$ has $n$ tiles, $G[x]$ is  precisely $\mathcal{I}^n$.  The subset $\gamma(x,y)$ is the set of all elements of  $\mathcal{I}^n$ that transform $x$ in $y$. Depending on the geometry of $x$ and $y$, this subset can contain more than one element. In the general case,  ($\Gamma_1$) states that if $g$ transforms $x$ in $y$ then $g^{-1}$ transforms $y$ in $x$. Axiom   ($\Gamma_2$) states that if $g$ transforms $y$ in $z$ and $h$ transforms $z$ in $w$, then $hg$ transforms $y$ in $w$. Finally,   ($\Gamma_3$) states that any element  $g$ of $\gamma(x,[x])$ transforms $x$ in one, and only one, tiling $y$ in the equivalence class of $x$. We call the elements of $\gamma(x,[x])$ the \emph{(admissible) pertubartions} of $x$.
The metric on the space of tilings of $\mathbb{R}^n$  will be induced by those of $G[x]$, with $x\in\mathcal{Y}$: two tilings will be close if they agree in a large ball about the origin up to a ``small" admissible perturbation. In view of this, we shall also need to  assume the following:
\begin{itemize}
\item[(G$_1$)] $d_{G{[x]}}$ is  right invariant, that is, $d_{G{[x]}}(fg,hg)=d_{G{[x]}}(f,h)$ for all $g,f,h\in G{[x]}$;
\item[(G$_2$)] for each pair of equivalence classes admitting representatives $x'\subseteq x$,
 there exists a homomorphism $\iota_{[x,x']}:{G{[x]}}\to{G{[x']}}$ such that, if $g\in\gamma(x,[x])$, then $\iota_{[x,x']}(g)\in \gamma(x',[x'])$ and $\iota_{[x,x']}(g)(x')\subseteq g(x)$;
\item[(G$_3$)] if $g\in G{[x]}$ and $x'\subseteq x$, then $\|\iota_{[x,x']}(g)\|_{G{[x']}}\leq \|g\|_{G{[x]}}$, where  $\|g\|_{G{[x]}}=d_{G{[x]}}(g,I\!d_{G{[x]}})$ is the \emph{size} of $g$;
\item[(G$_4$)] given  $x,y\in\mathcal{Y}$,  $G{[x, y]}:=\{g\in \gamma(x,[x]):\,g(x)\subseteq y\}$ is closed in $\gamma(x,[x])$.
\end{itemize}
Axiom (G$_2$) says that an admissible perturbation $g$ on a tiling $x$ induces an admissible perturbation on each patch $x'$ of $x$, which, by  (G$_3$), has size  less or equal than the size of $g$.
For notational convenience, we shall denote $\iota_{[x,x']}(g)$ by $g$ whenever  it is clear which  equivalence classes we are dealing with. Observe that $\iota_{[x,x']}$ only depends on the equivalence classes $[x]$ and $[x']$. The right invariance of $d_{G[x]}$ gives us for free the items (a) and (b) of the following lemma:
\begin{lem}\label{G}
For all $x\in \mathcal{Y}$, and $g,h\in G[x]$, we have:
      \begin{itemize}
        \item[(a)]  $\|g\|_{_{G[x]}}=\|g^{-1}\|_{_{G[x]}}$;
        \item[(b)]  $\|hg\|_{_{G[x]}}\leq \|g\|_{_{G[x]}}+\|h\|_{_{G[x]}};$
        \item[(c)] if $x'\subseteq x$, $\iota_{[x,x']}:G[x]\to G[x']$ is continuous.
      \end{itemize}
    \end{lem}
   \begin{proof}
    For instance,
     $$\|hg\|_{_{G[x]}}=d_{G[x]}(hg,I\!d_{G[x]})\leq d_{G[x]}(hg,g)+d_{G[x]}(g,I\!d_{G[x]})= \|h\|_{_{G[x]}}+\|g\|_{_{G[x]}}.$$ Item (c) is a direct consequence of (G$_3$).
   \end{proof}

We emphasize that, even when $x$ is a single tile, the group $G[x]$ does not have to be induced by a group of continuous transformations on $\mathbb{R}^n$. For example, with minor technical changes we could associate a set of labels to each tile and think on groups inducing actions on these sets of labels, giving in this way an alternative approach to handle with the fusion tiling spaces with local infinite complexity studied by Frank and Sadun \cite{FS2}. However, we need to impose a certain kind of continuity to the action of elements of $\gamma(x,[x])$ on the support of $x$. For that, consider the auxiliary set $\Theta$ of functions $\theta:]\sqrt{2},\infty[\times\R_0^+\to\R_0^+$ satisfying:
\begin{itemize}
\item[($\Theta_1$)]for each $s>\sqrt{2}$ and $b\geq 0$, the function $\theta_s(\cdot):=\theta(s,\cdot)$ is strictly increasing and the function $\theta^b(\cdot):=\theta(\cdot,b)$ is increasing;
\item[($\Theta_2$)] $\theta(s,a+b)\leq \theta(s,a) + \theta(s,b)$, for all $s>\sqrt{2}$, and $a,b\geq 0$;
\item[($\Theta_3$)] $\theta $ is continuous and $\theta(s,0)=0$, for all $s>\sqrt{2}$.
\end{itemize}
Assume that:
\begin{itemize}
  \item[(G$_5$)] there exists  $\theta\in\Theta$ such that, for all $x\in\mathcal{Y}$, $s>\sqrt{2}$ and  $g\in \gamma(x,[x])$ with $\theta(s,\|g\|_{G{[x]}})<\sqrt{2}/2$, we have:
 if $\mathrm{supp}( x)\subseteq \mathbb{R}^n\setminus B_s$, then $\mathrm{supp}( g(x))\subseteq \mathbb{R}^n\setminus B_{s-\theta(s,\|g\|_{G{[x]}})}$.

\end{itemize}
 We can think on $\theta(s,b)$ as the ``maximal Euclidean impact" that  perturbations $g$ of size $b$ induce on a ball of radius $s$ centered at the origin.
 Having the rigid motions in mind, we expect that this  impact strictly increases with the size of the perturbations and vanishes if and only if $b=0$.  For rigid motions, the relative position of $g(x)$ with respect to $x$  depends on the transformation $g$ but not  on the initial position of $x$ in the Euclidean space $\R^n$. However, for other transformations, such as homotheties, the relative position depends also on the initial position of $x$. That is why in the general case we need to consider the functions $\theta$ depending on two variables.
 Observe that the ``Euclidean impact"  induced by a given homothety increases with the radius $s$.
 The appearance of $\sqrt 2$ only has to do with the choice in \eqref{distance} of the maximal distance between tilings.

Before describe in detail some examples, let us establish the following useful lemma:
\begin{lem}\label{i}
  If $B_s\subseteq \mathrm{supp}( x)\cap \mathrm{supp}\big(g( x)\big)$, then, for all $\sqrt{2}<s'\leq s$ and $x'\subseteq x$ with $B_{s'}\subseteq  \mathrm{supp}( x')$, we have  $B_{s'-\theta(s',\|g\|_{G{[x]}})}\subseteq \mathrm{supp}\big(g( x')\big)$.
\end{lem}
\begin{proof}
   Take $\sqrt{2}<s'\leq s$ and a patch $x'$ of $x$  such that $B_s\subseteq \mathrm{supp}( x)\cap \mathrm{supp}\big(g(x)\big)$ and $B_{s'}\subseteq \mathrm{supp}( x')$. Set $t=\|g\|_{G{[x]}}$.
 Suppose that $ B_{s'-\theta(s',t)}$ is not contained in $ \mathrm{supp}\big(g(x')\big)$. In this case, since $B_s\subseteq \mathrm{supp}( x)\cap \mathrm{supp}\big(g(x)\big)$, there must exists a tile $D_i\subset \mathbb{R}^n\setminus B_{s'}$ in  $x\setminus x'$ such that $g_i(D_i)\cap   B_{s'-\theta(s',t)}\neq \emptyset$.  But this contradicts (G$_5$).
\end{proof}

\begin{eg}\label{exampleG}
Let $G$ be a group  equipped  with a right invariant metric $d_G$. Suppose that we have a continuous group action of $G$ on  $\mathbb{R}^n$.
 Define an equivalence relation $\sim$ on $\mathcal{Y}$ as follows:
 \begin{itemize}\item[(E$_1$)]
given two elements $x=\{D_j\}_{j\in J}$ and $x'=\{D'_k\}_{k\in K}$ of $\mathcal{Y}$, write $x\sim x'$ if there exist a bijection $\alpha:J\to K$ and  $g\in G$ such that, for each $j\in J$, $D'_{\alpha(j)}=g(D_j)$.
\end{itemize}
Let $\gamma(x,x')$ be the set of all such elements of $G$, and
 for each equivalence class $[x]$ put $G[x]=G$. If $x'\subseteq x$, set $\iota_{[x,x']}(g)=g$. Clearly, these choices satisfy (G$_1$)-(G$_4$). Let us consider two particular cases and find corresponding  functions $\theta\in \Theta$ satisfying (G$_5$):
 \begin{itemize}
 \item[(a)]  $G$ is the group of rigid motions (direct isommetries) of $\mathbb{R}^n$:
$$\mathcal{I}=\{g: \,g(\vec v)=R(\vec v)+\vec g,\,\textrm{with $R\in SO(n)$ and $\vec g\in\mathbb{R}^n$}\}.$$
 We can define on $\mathcal{I}$ a right invariant metric $d_{\mathcal{I}}$ by
 \begin{equation}\label{di}
   d_\mathcal{I}(g,h)=\max_{\|\vec v\|\leq 1}\|g ^{-1}(\vec v)-h^{-1}(\vec v)\|.
 \end{equation}
Given $g\in  \mathcal{I}$ with $g(\vec v)=R(\vec v)+\vec g$, we have $\|g\|_{\mathcal{I}}\geq \|\vec g\|$, and $B_{s-\|g\|_\mathcal{I}}\subseteq B_{s-\|\vec g\|}\subseteq g(B_s)$. Hence (G$_5$)  holds for $\theta(s,t)=t$.
\item[(b)] $G$ is the group of homothetic transformations   of $\mathbb{R}^n$:
\begin{equation*}
\label{dH}
\mathcal{H}=\{g: \,g(\vec v)=\lambda \vec v+\vec g,\,\textrm{with $\lambda>0$ and $\vec g\in\mathbb{R}^n$}\}.
\end{equation*}
Consider the  metric on $\mathcal{H}$ defined by: if $g(\vec v)=\lambda \vec v+\vec g$ and $h(\vec v)=\mu \vec v+\vec h$, then
$\tilde{d}_\mathcal{H}(g,h)=\max\big\{|\ln({\lambda}/{\mu})|,\|\vec g-\vec h\|\big\}.$ This equips $\mathcal{H}$ with a structure of topological group with respect to which the standard action of $\mathcal{H}$ on $\mathbb{R}^n$ is a continuous action.  Although $\tilde{d}_\mathcal{H}$ is not right invariant, we can construct a right invariant metric $d_\mathcal{H}$ on $\mathcal{H}$, topologically equivalent  to $\tilde{d}_\mathcal{H}$, as follows: consider the continuous function $F:\mathcal{H}\to [0,1]$ defined by $F(g)=\max\{1-\tilde{d}_\mathcal{H}(I\!d_\mathcal{H},g),0\}$ and set
\begin{equation}\label{dh}
d_\mathcal{H}(g,h)=\sup_{f\in \mathcal{H}}\big|F(gf)-F(hf)\big|.
\end{equation}
This construction is motivated by the standard proof of Birkhoff-Kakutani theorem (see \cite{MZ} for details). Now, it follows from the definition \eqref{dh} that $$\|g\|_\mathcal{H}\geq \tilde{d}_\mathcal{H}(I\!d_\mathcal{H},g)=\max\{|\ln (\lambda)|,\|\vec g\|\}$$ if $\tilde{d}_\mathcal{H}(I\!d_\mathcal{H},g)< 1$. In this case, it is clear that $$B_{s-(s\|g\|_\mathcal{H}+\|g\|_\mathcal{H})}\subseteq B_{s-(s|\ln(\lambda)|+\|\vec g\|)}\subseteq B_{\lambda s-\|\vec g\|}  \subseteq g(B_s).$$ Otherwise, if $\tilde{d}_\mathcal{H}(I\!d_\mathcal{H},g)\geq  1$, we have $\|g\|_\mathcal{H}=1$, and   $B_{s-(s\|g\|_\mathcal{H}+\|g\|_\mathcal{H})}=\emptyset.$
It is now easy to check that (G$_5$) holds for $\theta(s,t)=st+t$.

\end{itemize}

\end{eg}

\begin{eg}\label{exampleG2}
Again, let $G$ be a group  equipped  with a right invariant metric $d_G$. Suppose that we have a continuous group action of $G$ on  $\mathbb{R}^n$.
 Define an equivalence relation $\sim$ on $\mathcal{Y}$ as follows:
 \begin{itemize}
 \item[(E$_2$)]
given two elements $x=\{D_j\}_{j\in J}$ and $x'=\{D'_k\}_{k\in K}$ of $\mathcal{Y}$, write $x\sim x'$ if there exist a bijection $\alpha:J\to K$ and a collection $\{g_j\}_{j\in J}$ of elements in $G$ such that $D'_{\alpha(j)}=g_j(D_j)$  for each $j\in J$ and $\sup_{j\in J}\|g_j\|_G<\infty$.
\end{itemize}
Of course, if $x$ and $x'$ satisfy (E$_1$) then they also satisfy (E$_2$). Associate to $x=\{D_j\}_{j\in J}$ the group $G[x]\subseteq\prod_{j\in J}G$ of all maps ${g}:J\to G$ such that $\sup_{j\in J}\|g_j\|_G<\infty$, where $g_j={g}(j)$. The group multiplication  is given in the usual way: ${g}h(j)=g_jh_j$.  In this case, $\gamma(x,x')$ is the set of all $g\in G[x]$ satisfying (E$_2$).
  We can define on $G[x]$ a right invariant metric $d_{G[x]}$ by
 \begin{equation}\label{dP}
   d_{G[x]}(g,h)=\sup_{j\in J}d_G(g_j,h_j).
 \end{equation}
 For $x'\subseteq x$, with $x'=\{D_j\}_{j\in J'}$ and $J'\subseteq J$, set $\iota_{[x,x']}(g)=g_{|J'}$. These choices satisfy (G$_1$)-(G$_4$). Again, if $G=\mathcal{I}$ then (G$_5$) holds for $\theta(s,t)=t$ and if  $G=\mathcal{H}$ then (G$_5$) holds for $\theta(s,t)=st+t$.
 \end{eg}

Given an element $x'\in\mathcal{Y}$ with bounded support and $r>0$  with $B_{1/r}\subseteq \mathrm{supp}(x')$, set
\begin{align*}
\Delta(x',r)=\inf\{s>0: \,\mathrm{supp}\big(g(x'')\big)\subseteq  B_s\,\,\,\text{for all}\,& g\in\gamma(x'',[x'']) \,\text{and}\,x''\subseteq x'\,\\& \text{with}\,B_{1/r}\cap \mathrm{supp}\big(g(x'')\big)\neq \emptyset \}.\end{align*}
So, the ball of radius $\Delta(x',r)$ about the origin contains all the admissible copies of $x'$ (and of its patches) that intersect $B_{1/r}$.
Taking account $(\Gamma_1)$ and $(\Gamma_2)$, observe that $\Delta(x',r)=\Delta(y',r)$ if $y'\in [x']$ and  $B_{1/r}$ is contained in the supports of $x'$ and $y'$. Moreover, if $x''\subseteq x'$ and $s\leq r$, then $\Delta(x'',s)\leq \Delta(x',r)$. Henceforth we assume that
\begin{itemize}
\item[(G$_6$)]$\Delta(x',r)< \infty$ for all  $x'\in\mathcal{Y}$ with bounded support and $r>0$.
\end{itemize}
\begin{rem}\label{rem1}
In the setting of Examples \ref{G} and \ref{exampleG2}, it is clear that (G$_6$) holds if $G=\mathcal{I}$. However, in general it does not hold when $G=\mathcal{H}$, unless we demand that all tiles in $\mathcal{Y}$ have diameter less or equal than a certain $M>0$.
\end{rem}

Denote by $Y$ the set of all tilings of $\mathbb{R}^n$ in $\mathcal{Y}$ and, for $x,y\in Y$,  define
  \begin{align}
  \nonumber  d(x,y)=\inf\Big\{\{\sqrt{2}/2\}\cup&\{0<r<\sqrt{2}/2: \,\textrm{exist $x'\in x[[B_{1/r}]]$,  $\,y'\in y[[B_{1/r}]],$}\\ &\textrm{and $g\in \gamma(x',y')\subseteq{G[x']}$ with $\theta(\Delta(x',r),\|g\|_{_{G[x']}})\leq r$}\}\Big\}.\label{distance}
  \end{align}
\begin{prop}
$(Y,d)$ is a metric space.
\end{prop}
\begin{proof}
  Clearly $d$ is non-negative and symmetric. Let us prove that the triangle inequality holds. Take $x,y,z\in Y$ with $0<d(x,y)\leq d(y,z)$ and
  $d(x,y)+d(y,z)<\sqrt{2}/2.$  Choose $\epsilon>0$ such that $\epsilon+d(x,y)+d(y,z)<\sqrt2/2,$
  and put $a=d(x,y)+\epsilon/2$ and $b=d(y,z)+\epsilon/2$. Then, by definition of $d$, there are $x'\in x[[B_{1/a}]]$, $y'\in y[[B_{1/a}]]$ and $g\in G[x']$ with $\theta(\Delta(x',a),\|g\|_{_{G[x']}})\leq a$ such that $g(x')=y'$. Similarly, there are $y''\in y[[B_{1/b}]]$, $z''\in z[[B_{1/b}]]$, $h\in G[y'']$ with $\theta(\Delta(y'',b),\|h\|_{_{G[y'']}})\leq b$ such that $h(y'')=z''$.

  Let $y_0=y'\cap y''\in y[[B_{1/b}]]$. By (G$_2$), we can define  $x_0=g^{-1}(y_0)$, $z_0=h(y_0)$, and we have $x_0\subseteq x'$ and $z_0\subseteq z''$.
Put $c=a+b$. Since $0<c<\sqrt2/2$, and taking account ($\Theta_1$), it turns out that
\begin{align}\label{g}\frac1b\geq \frac1c+a\geq \frac1c + \theta(\Delta(x',a),\|g\|_{_{G[x']}})\geq \frac1c + \theta(1/b,\|g\|_{_{G[x']}}).\end{align}  Then, by Lemma \ref{i} and (a) of Lemma \ref{G}, it follows from \eqref{g} that $x_0\in x[[B_{1/c}]]$. On the other hand, if $y''\setminus y_0\neq \emptyset$, we must have $\mathrm{supp}(y''\setminus y_0)\subset \mathbb{R}^n\setminus B_{1/a}$ and, by ($\Theta_1$) and (G$_3$),
 \begin{align}\label{h}\frac1a\geq \frac1c+b\geq \frac1c + \theta(\Delta(y'',b),\|h\|_{_{G[y'']}})\geq \frac 1c+\theta(1/a,\|h\|_{_{G[y''\setminus y_0]}}).\end{align}
Then, by (G$_5$), it follows from  \eqref{h} that $\mathrm{supp}\big(h(y''\setminus  y_0)\big)\subset \mathbb{R}^n\setminus B_{1/c}.$ Hence, since $h(y'')\in z[[B_{1/b}]]\subseteq z[[B_{1/c}]]$, we see that $h(y_0)\in z[[B_{1/c}]]$. If $y''\setminus y_0=\emptyset$, then $y_0=y''$ and   $h(y_0) \in z[[B_{1/c}]]$.

Since $hg(x_0)=z_0$ and, by (b) of Lemma \ref{G},
$\|hg\|_{_{G[x_0]}}\leq \|g\|_{_{G[x_0]}}+\|h\|_{_{G[x_0]}},$ then, by ($\Theta_1$) and ($\Theta_2$), $$\theta(\Delta(x_0,c),\|hg\|_{_{G[x_0]}})\leq \theta(\Delta(x',a),\|g\|_{_{G[x']}})+\theta(\Delta(y'',b),\|h\|_{_{G[y'']}})\leq a+b,$$ and thus
$d(x,z)\leq a+b =d(x,y)+d(y,z)+\epsilon,$ with $\epsilon >0$ arbitrarily small. Hence the triangle inequality holds.

Finally, we want to prove that
$d(x,y)=0$ if, and only if, $x=y$. Clearly, if $x=y$ then $d(x,y)=0$. Assume now, for a contradiction, that $x\neq y$ and $d(x,y)=0$. Take  a patch $x'\subset x$, with $\mathrm{supp}(x')\subset B_{1/r}$ for some $r>0$, such that $x'\nsubseteq y$. On the other hand, by definition of $d$, for each $\delta>0$ there are $x_\delta\in x[[B_{1/\delta}]]$, $y_\delta\in y[[B_{1/\delta}]]$ and $g_\delta\in G[x_\delta]$, with $\theta\big(\Delta(x_\delta,\delta),\|g\|_{_{G[x_\delta]}}\big)<\delta$, such that $y_\delta=g_\delta(x_\delta)$. Take $\delta< r$. In this case, since $x'\subset x_\delta$, by (G$_2$) we have $g_\delta(x')\subset y_\delta\subset y$. Hence $g_\delta\in G{[x', y]}$.   At the same time,  by ($\Theta_1$) and (G$_3$), we have, for any $\sqrt{2}<s< \Delta(x_\delta,\delta) $,
 $$\theta\big(s,\|g_\delta\|_{_{G[x']}}\big)\leq \theta\big(s,\|g_\delta\|_{_{G[x_\delta]}}\big)\leq \theta\big(\Delta(x_\delta,\delta),\|g_\delta\|_{_{G[x_\delta]}}\big)< \delta,$$
 which, by ($\Theta_1$) and ($\Theta_3$),  implies that $g_\delta\to I\!d_{G{[x']}}$ as $\delta\to 0$. By (G$_4$),  this would imply that $I\!d_{G{[x']}}\in G[x',y]$, that is, $x'\subset y$, which is a contradiction.
\end{proof}
\begin{prop}\label{complete}
 $(Y,d)$ is complete if, for all $x\in\mathcal{Y}$ with bounded support, the following conditions hold:
 \begin{itemize}
 \item[(C$_1$)] $\gamma(x,[x])\subseteq G[x]$ is complete with respect to the restriction of $d_{G[x]}$.
  \end{itemize}

\end{prop}

\begin{proof}
  Consider a Cauchy sequence $\{x_n\}_{n\in \mathbb{N}}$ of tilings of $\mathbb{R}^n$ in $Y$ and consider a sequence $\{s_n\}_{\in \mathbb{N}}$ such that $d(x_n,x_{n+1})<s_n$. By definition of $d$, for each $n$ there are $x'_n\in x_n[[B_{1/{s_n}}]]$ and $g_n\in G[x'_n]$
  such that $\theta(\Delta(x'_n,s_n),\|g_n\|_{_{G[x'_n]}})\leq s_n$  and $g_n(x'_n)\in x_{n+1}[[B_{1/{s_n}}]].$ Without loss of generality, by passing to a subsequence if necessary, we can assume that the sequence $\{s_n\}_{n\in\mathbb{N}}$ is rapidly decreasing so that $\sum s_n<\infty$ and $\Delta(x'_n,s_n)\leq 1/s_{n+1}$. In particular, we have $g_n(x'_n)\subseteq x'_{n+1}.$

  Given $n<m$ we define $h_{n,m}=g_mg_{m-1}\ldots g_{n+1}g_n\in \gamma(x'_n,[x'_n])\subseteq G[x'_n]$. Of course, here we are  denoting $\iota_{[x'_{n+k},x'_n]}(g_{n+k})$ by $g_{n+k}$. For $n<m'<m$, by right invariance we have
  \begin{align}\nonumber
  d_{G[x'_n]}(h_{n,m},h_{n,m'})&=d_{G[x'_n]}(g_mg_{m-1}\ldots g_{n+1}g_n,g_{m'}g_{m'-1}\ldots g_{n+1}g_n)\\& =d_{G[x'_n]}(g_mg_{m-1}\ldots g_{m'+1},I\!d_{G[x'_n]})\nonumber\\&=\|g_mg_{m-1}\ldots g_{m'+1}\|_{_{G[x'_n]}}\nonumber\\&\leq
 \sum_{i=m'+1}^m\|g_i\|_{_{G[x'_n]}}.\label{cona1}
  \end{align}
  On the other hand, since $\Delta(x'_n,s_n)<\Delta(x'_{n+k},s_{n+k})$ for all $k>0$, we have
  $$\sum_{i= n}^\infty\theta(\Delta(x'_n,s_n),\|g_i\|_{_{G[x'_n]}})\leq \sum_{i= n}^\infty s_i <\infty,$$
  then
  \begin{align*}
  \theta\big(\Delta(x'_n,s_n),\sum_{i=m'+1}^m\|g_i\|_{_{G[x'_n]}}\big)\leq \sum_{i=m'+1}^m
  \theta(\Delta(x'_n,s_n), \|g_i\|_{_{G[x'_n]}})\to 0\,\,\,\,\textrm{as ${m,m'\to \infty}$}.
  \end{align*}
 By $(\Theta_1)$ and $(\Theta_3)$, this implies that
 \begin{equation}\label{cona2} \sum_{i=m'+1}^m\|g_i\|_{_{G[x'_n]}}\to 0\,\,\,\,\textrm{as ${m,m'\to \infty}$}\end{equation}
From \eqref{cona1} and \eqref{cona2}, we see that $\{h_{n,m}\}_{m\in\mathbb{N}}$ is a Cauchy sequence. Hence, by the completeness of $\gamma(x'_n,[x'_n])$, $\{h_{n,m}\}_{m\in\mathbb{N}}$ is convergent to a certain element $h_n\in \gamma(x'_n,[x'_n])$.

   By right invariance of $d_{G[x'_n]}$, the right multiplication by an element is continuous in ${G[x'_n]}$. At the same time, by (c) of Lemma \ref{G}, $\iota_{[x'_{n+1},x'_n]}$ is a continuous homomorphism. Hence
   $h_n=h_{n+1}g_n$. Consequently,
  $h_n(x'_n)=h_{n+1}g_n(x'_n)\subseteq h_{n+1}(x'_{n+1}).$
  This implies that $\{h_n(x'_n)\}_{n\in\mathbb{N}}$ is an increasing sequence in $\mathcal{Y}$, so, by $(\mathcal{Y}_2)$, we can define a tiling $x=\bigcup_nh_n(x'_n)\in Y$. Next we prove that $\{x_n\}_{n\in\mathbb{N}}$ converges to $x$.

Since $1/s_n\leq \Delta(x'_n,s_n)$, it follows from  ($\Theta_1$) and   ($\Theta_2$)  that
\begin{align*}
   \nonumber\theta(1/s_n,\|h_{n,m}\|_{_{G[x'_n]}})\leq \sum_{i=n}^m  \theta(1/s_i,\|g_i\|_{_{G[x'_i]}})\leq \sum_{i=n}^m s_i,\label{qq}\end{align*} and consequently,  by continuity,
we have
$$\theta(1/s_n,\|h_{n,m}\|_{_{G[x'_n]}})\to \theta(1/s_n,\|h_{n}\|_{_{G[x'_n]}})\leq \sum_{i=n}^\infty s_i , \,\,\,\,\textrm{as ${m\to \infty}$.}$$
Hence,   by Lemma \ref{i},   $h_{n}(x'_n)\in x[[B_{t_{n}} ]],$  where $t_n=1/s_n-\sum_{i=n}^\infty s_i$. Moreover, since
$\Delta(x'_n,1/t_{n})\leq \Delta(x'_n,s_n)$, it follows that,
taking $m\to \infty$ and using the continuity of $\theta$,
 \begin{equation}\label{limt}
 \lim_{n\to\infty} \theta\big(\Delta(x'_n,1/t_n),\|h_{n}\|_{_{G[x'_n]}}\big)=0.
  \end{equation}
Finally, from \eqref{limt} we conclude that
$$d(x,x_n)\leq \max\big\{1/{t_n}, \theta\big(\Delta(x'_n,1/t_n),\|h_{n}\|_{_{G[x'_n]}}\big)\}\rightarrow 0, \,\,\,\,\textrm{as ${n\to \infty}$},$$ since $\lim_n t_n=\infty$. \end{proof}
Now, associate to each equivalence  class $[x]$ of elements in $Y$ an element  $$g_{[x]}\in \bigcap_{z\in [x]} \gamma(z,[x])\subseteq G[x]$$ in such a way that, if $x'\subseteq x$, $y'\subseteq{y}$ and $[x']=[y']$, then $\iota_{[x,x']}(g_{[x]})=\iota_{[y,y']}(g_{[y]})$. This defines a map
\begin{equation}\label{map}
g:Y\to Y, \,\,\,\,\textrm{$g(z)=g_{[x]}(z)$ if $z\in [x]$},
\end{equation}
and we have:
\begin{prop}\label{continuity}
  Suppose that, for each $x\in Y$, the left multiplication by $g_{[x]}\in G[x]$ is continuous in $(G[x],d_{G[x]})$. Moreover, assume that, for each $x\in Y$, there exists $\epsilon<1$ such that $\delta\theta(1/\delta,\|g_{[x]}\|_{G[x]})<\epsilon$ for all sufficiently small $\delta>0$. Then the map $g:(Y,d)\to(Y,d)$ defined by \eqref{map}  is continuous.
\end{prop}
\begin{proof}
  Consider a sequence $\{x_n\}_{n\in\mathbb{N}}$ of elements in $Y$ convergent to $y\in Y$. This means that, for each $n$ sufficiently large, there exist $s_n>0$, $x''_n\in x_n[[B_{1/s_n}]]$, $y''_n\in y[[B_{1/s_n}]]$
  and $g_n\in{G}[y''_n]$ such that $\lim s_n=0$, $y''_n=g_n(x''_n)$ and $\theta(\Delta(x''_n,s_n),\|g_n\|_{_{G[y''_n]}})\leq s_n$. Observe that, by ($\Theta_3$), we must have $\lim\|g_n\|_{_{G[y''_n]}}=0$.

   Given $\delta>0$, take $N_\delta,n_\delta>0$ such that, for all $n>n_\delta$,
  we have $1/s_n>N_\delta$ (hence $y''_n\in y[[B_{N_\delta}]]$) and
  $$1/\delta<N_\delta-\theta\big(\Delta(x''_n,s_n),\|g_n\|_{G[y''_n]} \big).$$
  Set $y'_\delta=\bigcap_{n>n_\delta}y''_n$ and $x'_n=g_n^{-1}(y'_\delta)$.
   By Lemma \ref{i}, $x'_n\in x_n[[B_{1/\delta}]]$.
  Since $[x'_n]=[y'_\delta]$, we can define $$h_\delta:=\iota[x_n,x'_n](g_{[x_n]})=\iota[y,y'_\delta](g_{[y]}).$$
 Observe that $h_\delta g_n h_\delta^{-1}(h_\delta(x'_n))=h_\delta(y'_\delta)$ and, by (G$_2$) and Lemma \ref{i}, $h_\delta(x'_n)\in g(x_n)[[B_{1/{\delta'}}]]$ and  $h_\delta(y'_\delta)\in g(y)[[B_{1/{\delta'}}]]$, with
 $1/\delta'=1/\delta-\theta(1/\delta,\|h_\delta\|_{_{G[y'_\delta]}})$, that is
 $$\delta'=\frac{\delta}{1-\delta\,\theta(1/\delta,\|h_\delta\|_{_{G[y'_\delta]}})}\leq \frac{\delta}{1-\delta\,\theta(1/\delta,\|g_{[y]}\|_{_{G[y]}})}\leq \frac{\delta}{1-\epsilon},$$
   for all sufficiently small $\delta>0$. Then
   \begin{align}\label{dn}\nonumber d(g(x_n),g(y))&\leq \max\left\{\delta',\theta\big(\Delta(h_\delta(x'_n),\delta') ,\|h_\delta g_n h_\delta^{-1}\|_{_{G[y'_\delta]}}\big)\right\}
   \\&= \max\left\{\delta',\theta\big(\Delta (y'_\delta,\delta') ,\|h_\delta g_n h_\delta^{-1}\|_{_{G[y'_\delta]}}\big)\right\}
   .\end{align}
Since $\lim\|g_n\|_{_{G[y''_n]}}=0$ and $\|g_n\|_{_{G[y'_\delta]}}\leq \|g_n\|_{_{G[y''_n]}}$, we also have $\lim\|g_n\|_{_{G[y'_\delta]}}=0$ and, consequently, $\lim g_n=I\!d_{G[y'_\delta]}$. On the other hand, left multiplication by $h_\delta$ is continuous and, by the right invariance of $d_{G[x'_\delta]}$, right multiplication by any element of $G[y'_\delta]$ is also continuous. Then $\lim h_\delta g_n h_\delta^{-1}= I\!d_{G[y'_\delta]}$. Hence, for $n$ sufficiently large, from  inequality \eqref{dn} we see that  $d(g(x_n),g(y))\leq \delta '$. But $\delta'\to 0$ as $\delta\to 0$. Hence
 $\{g(x_n)\}_{n\in\mathbb{N}}$ converges to $g(y)\in Y$.
\end{proof}

\begin{eg}\label{ex3} Consider the setting of Examples \ref{exampleG} and \ref{exampleG2}. Given $g\in G$, assume that $g\in \gamma(x,[x])$ whenever $x\in Y$.
 When $G=\mathcal{I}$,  it follows directly from the previous proposition that the isometry  $g\in \mathcal{I}$ defines a continuous transformation on $Y$.
On the other hand, when $G=\mathcal{H}$, and taking account Remark \ref{rem1}, $g$ must be a translation.
If $\|g\|_{\mathcal{H}}<1$, then it is also clear that $g$ defines a continuous transformation on $Y$. The conditions of the proposition are not fulfilled if $\|g\|_{\mathcal{H}}=1$. However, since $g$ can always be  written as the product of a finite number translations,  $g=g_1\ldots g_n$ with $\|g_i\|_{\mathcal{H}}<1$, we conclude that the corresponding transformation on $Y$, as the composition of a finite number of continuous transformations, is also continuous.
\end{eg}

Next we shall be concerned with the compactness of $(Y,d)$.
 Given a compact $K\subset \mathbb{R}^n$, we say that $x'\in\mathcal{Y}$ is \emph{$K$-minimal} if: $K\subseteq \mathrm{supp}(x')$; given $x''\in\mathcal{Y}$  such that $K\subseteq \mathrm{supp}(x'')$ and $x''\subseteq x'$, then $x''=x'$. We denote by ${\mathcal{Y}_{\min}}(K)$ the subset of all $K$-minimal elements of  $\mathcal{Y}$, which is obviously nonempty.

 \begin{prop}Suppose that $\mathcal{Y}$ satisfies:
\begin{itemize}
       \item[(C$_2$)] for all compact $K\subset \mathbb{R}^n$, there exists a finite subset $\mathcal{A}_K\subset \mathcal{Y}$ such that, for all $x'\in{\mathcal{Y}_{\min}}(K)$, there are $y'\in\mathcal{A}_K$ and $g\in G[y']$ with $x'=g(y')$;
       \item[(C$_3$)]  for all compact $K\subset \mathbb{R}^n$ and $x'\in {\mathcal{Y}_{\min}}(K)$, $$G_{x'}(K)=\left\{g\in \gamma(x',[x']): K\subseteq \mathrm{supp}(g(x'))\right\}$$ is compact.
 \end{itemize}
Then the metric space    $(Y,d)$ is compact if it is complete.
 \end{prop}
 \begin{proof}
   Let $\{x_n\}_{n\in\mathbb{N}}$ be a sequence in $Y$. We want to extract a subsequence  convergent to an element $x\in Y$. This is done by a standard diagonalization argument. Fix an increasing sequence of positive real numbers $\{r_n\}_{n\in\mathbb{N}}$ such that $\lim r_n=\infty$. Denote by $x'_{n}(r_k)$   the unique  $B_{r_{\!k}}$-minimal patch of $x_{n}$.
  By (C$_2$), since $\mathcal{A}_{\!B_{r_1}}$ is finite, there is an infinite subset $I_1\subseteq \mathbb{N}$, with $m_1=\min I_1$, such that, for each $n\in I_1$,  there is $g_{1,n}\in G_{x_{m_1}'\!(r_1)}(B_{r_1})$ satisfying $x'_{n}(r_1)=g_{1,n}\big(x'_{m_1}(r_1)\big)$.
 However, by (C$_3$), we can assume, by taking a subsequence if necessary, that the sequence $\{g_{1,n}\}_{n\in I_1}$ converges to some $g_1\in G_{x_{m_1}'\!(r_1)}(B_{r_1})$. Hence, there
 exists  $n_1\in I_1$  such that for $n>n_1$ and $n\in I_1$, we have
  $$\theta\big(\Delta(x_{n_1}'(r_1),1/r_1),\|g_{1,n}g_{1,n_1}^{-1}\|_{_{G[x_{m_1}'(r_1)]}}\big) \leq 1/{r_1}.$$ We can now proceed recursively in order to obtain, for each $k\in \mathbb{N}$, an infinite set $I_k$, with $I_k\subseteq I_{k-1}\subseteq \ldots \subseteq I_2\subseteq I_1\subseteq \mathbb{N}$,  such that, for each $n\in I_k$, there is $g_{k,n}\in G_{x'_{m_k}\!(r_k)}(B_{r_k})$  with $x'_{n}(r_k)=g_{k,n}\left(x'_{m_k}(r_k)\right)$, where $m_k=\min I_k$. Again, without loss of generality, we can assume that the sequence $\{g_{k,n}\}_{n\in I_k}$ converges to some $g_k\in G_{x'_{m_k}\!(r_k)}(B_{r_k})$, which means that  there exists  $n_k\in I_k$, with $n_k>n_{k-1}>\ldots >n_2>n_1$, such that, for $n>n_k$ and $n\in I_k$,  we have
  $$\theta\big(\Delta(x_{n_k}'(r_k),1/r_k),\|g_{k,n}g_{k,n_k}^{-1}\|_{_{G[x_{m_k}'(r_k)]}}\big)\leq  1/{r_k}.$$
Observe that $x'_{n_{k+1}}(r_k)=g_{k,{n_{k+1}}}g_{k,n_k}^{-1}\left(x'_{n_{k}}(r_k)\right).$
 Then
  \begin{align}\label{dec}
 d(x_{n_k},x_{n_{k+1}})& \leq \max\left\{1/r_k,\theta\big(\Delta(x_{n_k}'(r_k),r_k),\|g_{k,{n_{k+1}}}g_{k,n_k}^{-1}\|_{_{G[x_{m_k}'(r_k)]}}\big)    \right\}    \leq 1/{r_k}.\end{align}
 Consider the infinite subset $I'=\{n_1,n_2,\ldots\}\subseteq \mathbb{N}$. From \eqref{dec} we see that
 that $\{x_n\}_{n\in I'}$ is a Cauchy subsequence of $\{x_n\}_{n\in\mathbb{N}}$. Since $Y$ is complete, $\{x_n\}_{n\in I'}$ admits a subsequence  that converges to some $x\in Y$.
\end{proof}
\begin{eg}\label{ex4} Let $\mathcal{F}$ be a finite set of tiles and fix $G=\mathcal{I}$. Let $y$ be a tiling of $\mathbb{R}^n$  by direct isometric copies of tiles in $\mathcal{F}$. Define $\mathcal{Y}$ as follows: $x'\in \mathcal{Y}$ if $x'$  is a direct isometric copy of some $y'\subset y$. It is clear that $ \mathcal{Y}$ satisfies ($\mathcal{Y}_1$) and  ($\mathcal{Y}_2$).  With respect to the equivalence relation of Example \ref{exampleG}, assume that
 the number  $\#\mathcal{F}^{(2)}_{E_1}$ of equivalence classes $[x']$, with $x'\in\mathcal{Y}$ composed by two tiles in $\mathcal{F}$ and $\mathrm{supp}(x')$ connected,  is finite.
 Following \cite{LW, Rob}, in this case we say that $Y$ has \emph{finite local complexity under isometries}. For all $x\in\mathcal{Y}$,  $\gamma(x,[x])=\mathcal{I}$, which is complete with respect to  $d_{\mathcal I}$. Condition (C$_3$)  follows from the  continuity of the group action of $\mathcal{I}$ on $\mathbb{R}^n$ and
  (C$_2$) is a consequence of the finite local complexity under isometries property. Then  $(Y,d)$ is compact and the action of $\mathcal{I}$ on $Y$ is continuous. The space of classical Penrose kite and dart tilings of $\mathbb{R}^2$ fits in this case.

\begin{eg}
Let $\mathcal{F}$ be a finite set of tiles and fix $G=\mathcal{I}$. Given a subset $S\subset \mathbb{R}^n$, denote by $X_\mathcal{F}(S)$ the set of all tings of $S$ by direct isometric copies of tiles in $\mathcal{F}$. Define $\mathcal{Y}=\bigcup_SX_\mathcal{F}(S)$, which satisfies  ($\mathcal{Y}_1$) and  ($\mathcal{Y}_2$). With respect to the equivalence relation of Example \ref{exampleG2},  $\#\mathcal{F}^{(2)}_{E_2}$ is automatically finite, hence (C$_2$) is satisfied.
  For each $x\in \mathcal{Y}$ with bounded support, the completeness of $\gamma(x,[x])$ follows from  the continuity of the group action of $\mathcal{I}$ on $\mathbb{R}^n$ and from  the completeness of $\mathcal{I}[x]$ with respect to the metric $d_{\mathcal{I}[x]}$ given by \eqref{dP}.
 Again, condition (C$_3$)  follows from the  continuity of the group action of $\mathcal{I}$ on $\mathbb{R}^n$.
  Then, $(Y,d)$ is compact and the action of $\mathcal{I}$ on $Y$ is continuous.
\end{eg}

\end{eg}
\begin{eg} Let $\mathcal{F}$ be a finite set of tiles and fix $G=\mathcal{H}$. Given a subset $S\subseteq \mathbb{R}^n$, denote by $X_\mathcal{F}(S)$ the set of all tilings of $S$ by homothetic  copies of tiles in $\mathcal{F}$ with scale factor $\lambda \in [\frac12,1]$. Take $\mathcal{Y}= \bigcup_{S} X_\mathcal{F}(S)$, which satisfies ($\mathcal{Y}_1$) and  ($\mathcal{Y}_2$). Consider the setting of Example \ref{exampleG2}. For each $x\in \mathcal{Y}$ with bounded support, the completeness of $\gamma(x,[x])$ follows from  the continuity of the group action of $\mathcal{H}$ on $\mathbb{R}^n$ and from  the completeness of $\mathcal{H}[x]$ with respect to the metric $d_{\mathcal{H}[x]}$ given by \eqref{dP}. Since $\mathcal{F}$ is finite and $\lambda$ takes values in a compact interval, (C$_2$) holds. Finally, (C$_3$) follows from the continuity of the group action of $\mathcal{H}$ on $\mathbb{R}^n$.
Hence $(Y,d)$ is compact and the action of the translation group on $Y$ is continuous.
\end{eg}
\begin{eg}\label{QP} Let $Q$ and $P$ be the hypercubes in $\mathbb{R}^n$, centered at the origin, with edges $1$ and $1/3$, respectively. Let $\mathcal{T}$ be the group of translations. If $T\in \mathcal{T}$ is the translation by the vector $\vec{t}$, then $\|T\|_\mathcal{T}=\|\vec t\|$.  Given a subset $S\subseteq \mathbb{R}^n$, let $ \mathcal{X}_{Q,P}(S)$ the set of all tilings of $S$ by  two equivalent classes of tiles: tiles of the form  $T(P)$, with $T\in\mathcal{T}$; tiles of the form $T(Q)\setminus T'(P)$, with $\|T-T'\|_\mathcal{T}\leq 1/6$.
 Set $\mathcal{Y}=\bigcup_S\mathcal{X}_{Q,P}(S)$,  $G[P]=\mathcal{T}$ and $G[Q\setminus P]=\mathcal{T}\times \mathcal{T}$. Consider on $G[Q\setminus P]$  the distance $$d_{\mathcal{T}\times\mathcal{T}}\big((T,T'),(T_1,T_1')\big)=\max\{d_\mathcal{T}(T,T_1),d_\mathcal{T}(T',T'_1)\}.$$ The elements of $\mathcal{T}$ act on the tiles of the first class in the natural way, and an element $(T_1,T_1')\in \mathcal{T}\times \mathcal{T}$ acts on a tile $T(Q)\setminus T'(P)$ of the second class by $(T_1,T_1')(T(Q)\setminus T'(P))=T_1T(Q)\setminus T_1'T'(P)$. In particular,
$$\gamma \big(T(Q)\setminus T'(P), \big[ T(Q)\setminus T'(P)]\big)= \big\{(T_1,T'_1)\in\mathcal{T}\times \mathcal{T}:\, \|T_1T-T'_1T'\|_\mathcal{T}\leq 1/6\big\}.$$
Extend in the natural way this equivalence relation  to all $\mathcal{Y}$ (see Figure  \ref{qua}), similarly to (E$_2$). In particular, given $x=\{D_i\}_{i\in I}$ and $x'=\{D'_j\}_{j\in J}$ in $\mathcal{Y}$, then $x\sim x'$ if there exists $\alpha:I\to J$ such that, for each $i\in I$, $D'_{\alpha(i)}=g_i(D_i)$ for some $g_i\in G[D_i]$, with $\sup_{i\in I}\|g_i\|_{G[D_i]}<\infty$. Thus, $G[x]\subseteq \prod_{i\in I} G[D_i]$ is the subgroup of all such collections $g=\{g_i\}_{i\in I}$ provided with the bi-invariant distance $d_{G[x]}(g,g')=\sup_{i\in I}d_{G[D_i]}(g_i,g'_i)$. These choices satisfy (G$_1$)-(G$_6$)  for $\theta(s,t)=t$. Fix on $Y$ the distance $d$ defined by \eqref{distance}. Clearly $\mathcal{T}$ acts continuously on $(Y,d)$.
The completeness of $\gamma(x,[x])$  follows from the continuity of the action of $\mathcal{T}$ on $\mathbb{R}^n$ and from the completness of $G[x]$. (C$_2$) follows from the closed restriction   $\|T-T'\|_\mathcal{T}\leq 1/6$.   Condition (C$_3$)  follows from the  continuity of the group action of $\mathcal{T}$ on $\mathbb{R}^n$.
 Hence $(Y,d)$ is compact.

\begin{figure}[!htb]\label{qua}
\psset{xunit=1.0cm,yunit=1.0cm,algebraic=true,dotstyle=*,dotsize=3pt 0,linewidth=0.4pt,arrowsize=3pt 2,arrowinset=0.25}
\begin{pspicture*}(-.5,-1.2)(5,2.9)
\psline(0,0)(1,0)
\psline(0,0)(0,1)
\psline(1,0)(1,1)
\psline(1,1)(0,1)
\psline(0,1)(0,0)
\psline(0,0)(1,0)
\psline(0.33,0.67)(0.67,0.67)
\psline(0.67,0.67)(0.67,0.33)
\psline(0.67,0.33)(0.33,0.33)
\psline(0.33,0.33)(0.33,0.67)
\psline(2,0.5)(1,0.5)
\psline(1,-0.5)(2,-0.5)
\psline(1,-0.5)(1,0.5)
\psline(2,-0.5)(2,0.5)
\psline(1.43,0.25)(1.76,0.25)
\psline(1.43,-0.08)(1.43,0.25)
\psline(1.76,0.25)(1.76,-0.08)
\psline(1.76,-0.08)(1.43,-0.08)
\psline(1,0.83)(1.33,0.83)
\psline(1.33,0.83)(1.33,0.5)
\psline(1.33,0.5)(1,0.5)
\psline(1.67,0.83)(1.67,0.5)
\psline(1.33,0.83)(1.67,0.83)
\psline(2,1.83)(1,1.83)
\psline(1,0.83)(1,1.83)
\psline(1,0.83)(2,0.83)
\psline(2,0.83)(2,1.83)
\psline(1.22,1.41)(1.56,1.41)
\psline(1.56,1.41)(1.56,1.07)
\psline(1.56,1.07)(1.22,1.07)
\psline(1.22,1.07)(1.22,1.41)
\psline(0.67,-0.33)(0.67,0)
\psline(1,-0.33)(0.67,-0.33)
\psline(0.67,-1)(0.67,0)
\psline(-0.33,-1)(0.67,-1)
\psline(0.67,0)(-0.33,0)
\psline(-0.33,-1)(-0.33,0)
\psline(0.31,-0.6)(-0.02,-0.6)
\psline(0.31,-0.27)(0.31,-0.6)
\psline(-0.02,-0.6)(-0.02,-0.27)
\psline(-0.02,-0.27)(0.31,-0.27)
\psline(2.81,-0.02)(2.81,0.98)
\psline(3.81,0.98)(2.81,0.98)
\psline(2.81,-0.02)(3.81,-0.02)
\psline(3.81,-0.02)(3.81,0.98)
\psline(3.1,0.25)(3.1,0.58)
\psline(3.43,0.58)(3.43,0.25)
\psline(3.1,0.58)(3.43,0.58)
\psline(3.43,0.25)(3.1,0.25)
\psline(4.15,0.97)(4.15,0.64)
\psline(4.81,0.64)(3.81,0.64)
\psline(3.81,-0.36)(4.81,-0.36)
\psline(4.81,-0.36)(4.81,0.64)
\psline(3.81,-0.36)(3.81,0.64)
\psline(4.15,0.98)(4.15,0.65)
\psline(4.48,0.98)(4.48,0.65)
\psline(3.81,-0.35)(3.48,-0.35)
\psline(3.48,-0.35)(3.48,-0.02)
\psline(3.48,-1.02)(3.48,-0.02)
\psline(2.48,-1.02)(2.48,-0.02)
\psline(3.48,-0.02)(2.48,-0.02)
\psline(2.48,-1.02)(3.48,-1.02)
\psline(4.46,0.23)(4.46,-0.1)
\psline(4.13,0.23)(4.46,0.23)
\psline(4.13,-0.1)(4.13,0.23)
\psline(4.46,-0.1)(4.13,-0.1)
\psline(3.12,-0.42)(3.12,-0.75)
\psline(2.79,-0.42)(3.12,-0.42)
\psline(2.79,-0.75)(2.79,-0.42)
\psline(3.12,-0.75)(2.79,-0.75)
\psline(3.44,0.98)(3.44,1.98)
\psline(4.44,0.98)(4.44,1.98)
\psline(3.44,0.98)(4.44,0.98)
\psline(4.44,1.98)(3.44,1.98)
\psline(3.76,1.43)(3.76,1.76)
\psline(4.09,1.43)(3.76,1.43)
\psline(3.76,1.76)(4.09,1.76)
\psline(4.09,1.76)(4.09,1.43)
\psline(4.48,0.98)(4.15,0.98)
\end{pspicture*}
 \caption{Two equivalent elements of $\mathcal{Y}$ in Example \ref{QP}.}
 \end{figure}
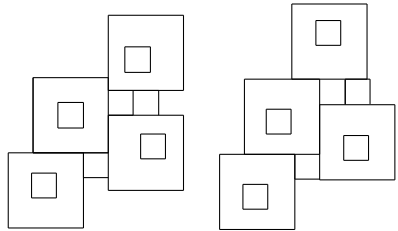

\end{eg}
\section{Almost periodicity and local isomorphism}\label{isomorphism}
Recall that, in a topological dynamical system, a point $x$ is \emph{almost periodic} if, for every neighborhood $U$ of $x$, the set  of ``return times" to $U$  is \emph{relatively dense}. For tiling spaces, almost periodicity implies  the \emph{local isomorphism} property, which is
 a property that a tiling of an Euclidean space might have which  expresses a certain ``regularity".
Almost periodicity is a necessary and sufficient condition to the \emph{minimality} of an orbit closure (Gottshalk's theorem \cite{Go}). An application of Zorn's lemma shows that every compact dynamical system admits a minimal invariant subset. It follows that a compact dynamical system always admits an almost periodic point.  In this section we discuss the relation between almost periodicity and local isomorphism property within our setting.

\begin{defn}
Let $(Y,d)$ be the metric space of all  tilings of $\mathbb{R}^n$ in $\mathcal{Y}$, where $d$ is defined by \eqref{distance}.  The tiling  $y\in Y$ is said to satisfy the \emph{local isomorphism} property if for every patch $y'$ of $y$ with bounded support there is some $r(y')>0$ such that, for every ball $B$ of $\mathbb{R}^n$ with radius $r(y')$, there exists $g\in\gamma(y',[y'])$ such that $\mathrm{supp}\big(g(y')\big)\subseteq B$ and $g(y')\subset y$.
\end{defn}

Observe that in the definitions of local isomorphism property given in \cite{GS,Rad}  only isometric copies of $y'$ are allowed. Let $\mathcal{T}$ be the group of the translations in $\mathbb{R}^n$. Assume that
 $\mathcal{T}\subseteq \gamma(y',[y'])$ for all $y'\in \mathcal{Y}$ and that $\mathcal{T}$ acts continuously on $(Y,d)$. Consider the topological dynamical system $(Y, \mathcal{T})$. Given $y\in Y$,
define the \emph{return set} of $y$ to an open subset $U\subset Y$ as
$$R(y,U)=\{\vec{t}\in \mathbb{R}^n:\, T_{-\vec{t}}\,(y)\in U\}.$$
 A subset $R\subset \mathbb{R}^n$ is \emph{relatively dense} if there is an $r>0$ such that every ball of radius $r$ in $\mathbb{R}^n$ intersects $R$. In this case we also say that $R$ is $r$-\emph{dense}. A tiling $y\in Y$ is \emph{almost periodic} if $R(y,U)$ is relatively dense for every open  $U\subset Y$ with $R(y,U)\neq \emptyset$.

 Now, assume that $y\in Y$ is almost periodic.
Take a patch $y'$ of $y$ and $\epsilon>0$. For  $\delta>0$ with $\mathrm{supp}(y')\subset B_{1/\delta}$,
consider the open set $U_{\delta}=\{x\in Y:\, d(x,y)<\delta\}$.
The set  $R(y,U_\delta)$ is $r_{\delta}$-dense for some $r_\delta>0$. Hence, given a ball $B$ of radius $r_\delta$ centered at a point $P\in \mathbb{R}^n$, there is $\vec t\in B$ such that $d(T_{-\vec t}\,(y),y)<\delta$, that is, there are
 $y''\in y[[B_{1/\delta}]]$,  $z''\in T_{-\vec t}\,(y)[[B_{1/\delta}]]$, and $g\in \gamma(y'',z'')$ with $\theta(\Delta(y'',\delta),\|g\|_{_{G[y'']}})\leq \delta$. By (G$_6$), $\theta(\Delta(y'',\delta)<\infty$ and we have $$\mathrm{supp}(g(y'))\subset B_{\Delta(y'',\delta)}.$$
Set  $r_{\delta}(y')=r_\delta+\Delta(y'',\delta).$
Since $y'\subset y''$,   $T_{\vec t}\,g(y')\subset y$ and its support is contained in the ball of radius $r_\delta(y')$ centered at $P$.
For $\delta$ sufficiently small, we must have $\|g\|_{G[y']}<\epsilon$. This shows that:
\begin{prop}\label{almost per}
  Let $(Y,d)$ be the metric space of all  tilings of $\mathbb{R}^n$ in $\mathcal{Y}$, where $d$ is defined by \eqref{distance}.  If the tiling  $y\in Y$ is almost periodic, then for every patch $y'$ of $y$ with bounded support and any $\epsilon>0$, there is some $r_{\epsilon}(y')>0$ such that, for every ball $B$ of $\mathbb{R}^n$ with radius $r_\epsilon(y')$, there exists $g\in\gamma(y',[y'])$ with $\|g\|_{G[y']}<\epsilon$ and $\vec t\in \mathbb{R}^n$ such that $\mathrm{supp}\big(T_{\vec t}g(y')\big)\subseteq B$ and $T_{\vec t}g(y')\subset y$.
\end{prop}
From this we see that almost periodicity is slight stronger than local isomorphism property, since in the first case the copies are obtained by small perturbations up to translations of the initial patch. Of course, in some contexts they are indeed equivalent (for instance, for tiling spaces with finite local complexity under translations). In the lack of almost periodicity, we can still find similar (but  weaker)  regularity, as we will see in the next sections.

Recall that if $G$ is a group acting continuously on a compact topological space $X$, then $X$ is said to be \emph{uniquely ergodic} if admits one and only one $G$-invariant Borel probability measure.
The support of a uniquely ergodic measure is minimal invariant. Hence,  under the unique ergodicity condition, almost all tilings in $Y$ are almost periodic.


\section{Brown's lemma and its topological dynamics version}\label{bltdv}
The main idea of Ramsey theory  is that arbitrarily large
sets cannot avoid a certain degree of ``regularity". This
is exemplarily illustrated by Gallai's theorem, a
multidimensional version of the seminal van der Waerden
theorem. De la Llave and Windsor \cite{LW} exploited an
application of the Furstenberg's topological multiple
recurrence theorem (which is a topological dynamics version
of the multidimensional version of van der Waerden's
theorem) to tilings. Another Ramsey-type result is
the so called Brown's lemma \cite{B1,B2}, which asserts
that any finite coloring of the natural numbers admits a
monochromatic piecewise syndetic set. In  Section
\ref{btiling} we shall give an application of this lemma to
tiling theory. Before that, let us fix a suitable statement
of Brown's lemma  and  establish   its topological dynamics
version.

\vspace{.20in}

Recall the following notions of largeness of subsets of a topological semigroup $G$:
\begin{itemize}
  \item[(a)] a  subset $S$ of $G$ is \emph{syndetic} if there
exists a compact $K\subseteq G$ so that for any $g\in G$, there exists $k\in K$ with
$gk\in S$;
\item[(b)] a subset $T$ of a topological semigroup $G$ is \emph{thick} if for any
compact set $K\subseteq G$ there exists $g\in G$ with $gK\subseteq T$;
\item[(c)] a subset of $G$ is \emph{piecewise syndetic}
if it is the intersection of a syndetic set and a thick set.
\end{itemize}
When $G=\mathbb{N}$, this means that $S$ is piecewise syndetic if $S$ contains arbitrarily long intervals  with bounded gaps. If $G=\mathbb{R}^n$, a subset is syndetic if and only if is relatively dense.

Brown \cite{B1,B2}  proved that any finite coloring of the
natural numbers admits a monochromatic piecewise syndetic
set (Brown's lemma), that is, there exist $q$, depending
only on the coloring, and  arbitrarily large monochromatic
sets $A=\{a_1<\ldots<a_n\}$ with $\max \{a_{i+1}- a_i\}\leq
q$. Since $q$  is independent of the size of the
monochromatic sets $A$, this fact is not an immediate
consequence of van der Waerden's theorem (see \cite{B3} for
a detailed discussion on the (non)relation between these
results).
 More recently, Hindman
and Strauss (see Theorem 4.40 in \cite{HS}) proved that a subset of a discrete semigroup $G$ is piecewise syndetic if and only if
its closure intersects the smallest ideal $K(\beta G)\subseteq \beta G$ of the Stone-\v{C}ech compactification $\beta G$ of $G$, which is
 never empty.  Since $G$ is dense in $\beta G$, for any finite coloring $G_1\cup\ldots\cup G_l$ of $G$ there exists a color $G_i$ whose closure intersects $K(\beta G)$, hence $G_i$ is a piecewise syndetic set.
In the $\mathbb{Z}^n$ case, this implies that:

\begin{lem}[Multidimensional Brown's lemma]\label{Brown}
Given a finite coloring of the integer lattice
$\mathbb{Z}^n$, there exists  $q\in \mathbb{N}$ and a color $A\subset \mathbb{Z}^n$ satisfying:
for any  finite subset $F=\{P_i\}_{i\in I}$ of
$\mathbb{Z}^n$ and any $k\in \mathbb{N}$, there exist
$\vec{t}\in \mathbb{Z}^n$ and a collection
$\{\vec{v}_i\}_{i\in I}$, with $\vec{v}_i\in
Q_q^n=\{\vec{u}=(u_1,\ldots, u_n)\in \mathbb{Z}^n:\, -q\leq
u_j\leq q \}$, such that $\{k P_i+\vec{t}+ \vec v_i\}_{i\in
I}\subset A$.
 \end{lem}
 \begin{proof}
   Let $A\subseteq \mathbb{Z}^n$ be a piecewise syndetic monochromatic set.  The set $A$ is the intersection of a syndetic set $S$ with a thick set $T$.  There exists $q>0$ such that for any $\vec u\in \mathbb{Z}^n$ there exists $\vec v\in Q_q^n$ such that $\vec u+\vec v\in S$. On the other hand, given $k\in\mathbb{N}$, there exists $\vec t\in\mathbb{Z}^n$ such that $k F+\vec t+ Q_q^n\subset T$, once $K=k F+ Q_q^n$ is compact. For each $i$, choose $\vec v_i \in  Q_q^n$ such that $k P_i+\vec t+\vec v_i\in  S$.
 \end{proof}

Let us compare this lemma with the well known (see
\cite{Fu}) multidimensional version of the van der
Waerden's theorem (also known as Gallai's theorem), which
asserts that, given a finite coloring of $\mathbb{Z}^n$,
any finite subset $F$ of $\mathbb{Z}^n$ has a monochromatic
homothetic copy $k F+\vec{t}$. However, it says nothing,
apart its existence, about the scale factor $k$. On the
other hand,   Lemma \ref{Brown} states that we can take any
$k$ once we allow  ``bounded perturbations" ($q$ only
depends on the coloring) in the structure of the homothetic
copies of $F$.

The Furstenberg's topological multiple recurrence theorem \cite{Fu} is a topological dynamics version of Gallai's theorem. The following is a topological dynamics version of Lemma \ref{Brown}:
\begin{lem}[Topological dynamics  multidimensional Brown's lemma]\label{BrownTop}
  Let $(X,d)$ be a compact metric space and $T_1,\ldots, T_l$ commuting homeomorphism of $X$. Given $\epsilon>0$, there exists $q\in\mathbb{N}$ satisfying: for each $k\in\mathbb{N}$, there exist $x_k\in X$ and a collection $\{\vec{u}_i\}_{i\in \{1,\ldots, l\}}$, with $\vec{u}_i=(u^i_1,\ldots, u^i_l)\in Q^l_q$, such that
  $$d(x_k,T_i^kT_1^{u^i_1}\ldots T_l^{u^i_l}(x_k))< \epsilon$$
  for all $i\in\{1,\ldots,l\}$. Moreover, $d(x_k,x_{k'})<\epsilon$ for all $k,k'\in \mathbb{N}$.
\end{lem}
\begin{proof}This is analogous to the standard proof of Furstenberg's topological multiple recurrence theorem from Gallai's theorem.
  Let $U_1,\ldots, U_r$ be a covering of $X$ by pairwise disjoint sets of less than $\epsilon$ diameter. Choose $y\in X$ and consider the  coloring $\mathbb{Z}^l=\bigcup_{i=1}^r C_i$ defined as follows:
  $(a_1,\ldots,a_l)\in C_i$ if
$T_1^{a_1}\ldots T_l^{a_l}(y)\in U_i$. According to Lemma
\ref{Brown}, we can fix $q/2\in \mathbb{N}$ and a cell $C_i$ satisfying: for
each $k\in\mathbb{N}$, there exists
$\vec{t}=(t_1,\ldots,t_l)\in \mathbb{Z}^l$ and a collection
$\{\vec{v}_i\}_{i\in \{0,\ldots, l\}}$, with
$\vec{v}_i=(v^i_1,\ldots,v^i_l)\in Q^l_{q/2}$, such that
 $C_i$ contains the homothetic
``$q/2$-distorted" copy $\{k
P_i+\vec{t}+\vec{v}_i:\,\,P_i\in F\}$ of
  $$F=\big\{P_0=(0,\ldots,0), P_1=(1,0,\ldots,0), P_2=(0,1,0,\ldots,0),\ldots, P_l=(0,\ldots,0,1)\big\}.$$ That is,
  $$\big\{(t_1+v^0_1,\ldots,t_l+v^0_l),(k+t_1+v^1_1,t_2+v^1_2, \ldots,t_l+v^1_l), \ldots, (t_1+v^l_1,t_2+v^l_2, \ldots,k+t_l+v^l_l)   \big\} $$
  is a subset of $C_i$.
  Let $x_k=T_1^{t_1+v^0_1}\ldots T_l^{t_l+v^0_l}(y)$.  We then have, with $\vec{u}_i=\vec v_i-  \vec v_0\in Q^l_q$,
  $$\{x_k,T_1^{k}T_1^{u_1^1}\ldots T_l^{u_l^1}(x_k),T_2^{k}T_1^{u_1^2}\ldots T_l^{u_l^2}(x_k),\ldots, T_l^{k}T_1^{u_1^l}\ldots T_l^{u_l^l}(x_k)\}\subseteq U_i.$$ The result follows now from the fact that the diameter of $U_i$ is less than $\epsilon$.
\end{proof}

\section{An application of Brown's Lemma to tiling}\label{btiling}
In \cite{LW}, the authors proved, by using the well known Furstenberg's multiple recurrence theorem,  that, for  the three standard metrics, given a tiling $y$ of $\mathbb{R}^n$ and a finite geometric pattern $F\subset \mathbb{R}^n$ of points, one can find a patch $y'$ of $y$ so that copies of $y'$ appear in $y$ ``nearly" centered on some scaled and translated version of the pattern (Theorems 2, 3 and 4 of \cite{LW}).  Taking account the general setting we have developed in Section \ref{I}, next we present a unified and generalized reformulation of these results (Theorem \ref{LW}). Furthermore, we give an application of Lemma \ref{BrownTop} to tiling theory (Theorem \ref{BT}).

\vspace{.20in}

Let $(Y,d)$ be the metric space of all  tilings of $\mathbb{R}^n$ in $\mathcal{Y}$, where $d$ is defined by \eqref{distance}.  Let $\mathcal{T}$ be the group of translations in $\mathbb{R}^n$ and denote by $T_{\vec{v}}$ the translation by the vector $\vec{v}\in\mathbb{R}^n$. Suppose that $\mathcal{T}\subseteq \gamma(y',[y'])$ for all $y'\in\mathcal{Y}$ and that each $g\in\mathcal{T}$ induces a map $g:(Y,d)\to (Y,d)$ in the conditions of Proposition \ref{continuity}.  In particular, $\mathcal{T}$ acts continuously on $(Y,d)$.

\begin{thm}\label{LW}
   Assume that $(Y,d)$ is compact.  Given $y\in Y$, $\epsilon>0$ and a finite subset $F=\{\vec{v}_1,\ldots, \vec{v}_l\}\subset \mathbb{R}^n$, there exist $k\in\mathbb{N}$  and a patch $y'$ with bounded support satisfying:
   \begin{itemize}
     \item[i)] the support of $y'$ contains the ball $B_{1/\epsilon}$ and $T_{\vec u}(y')\subset y$ for some $\vec u\in\mathbb{R}^n$;
         \item[ii)] for  each $\vec{v}_i\in F$ there exists $g_i\in G[y']$, with $\|g_i\|_{G[y']}<\epsilon$, such that
  $T_{k\vec{v}_i+\vec u}g_i(y')\subset y.$
   \end{itemize}
   \end{thm}

Again, this theorem says nothing about the scale factor $k$. The following theorem shows that we can take any $k$ once we allow  ``bounded perturbations"  in the structure of the homothetic copies of the geometric pattern  $F$. Since there exists no direct relation between Brown's lemma and van der Waerden theorem, one can not expect to  obtain Theorem \ref{BT} directly from Theorem \ref{LW}.
The proof of Theorem \ref{LW}, which we omit here, is a straightforward adaptation  of the arguments used in \cite{LW} and in the proof of Theorem \ref{BT}.

\begin{thm}\label{BT}
   Assume that $(Y,d)$ is compact.  Given $y\in Y$, $\epsilon>0$ and  $F=\{\vec{v}_1,\ldots, \vec{v}_l\}\subset \mathbb{R}^n$, there exists $q\in\mathbb{N}$  and a patch $y'$ with bounded support  satisfying:
   \begin{itemize}
     \item[i)] the support of $y'$ contains  the ball $B_{1/\epsilon}$ and $T_{\vec u}(y')\subset y$ for some $\vec u\in \mathbb{R}^n$;
     \item[ii)] for each $\lambda>0$ and  $\vec{v}_i\in F$, there exists $g_{\lambda,i}\in G[y']$,  with
$\|g_{\lambda,i}\|_{G[y']}<\epsilon$ and
     $$T_{\vec{w}_{\lambda,i}}T_{\lambda\vec{v}_i+\vec{t}_\lambda}g_{\lambda,i}(y')\subset y,$$ for some $\vec{w}_{\lambda,i}\in Q_q(F)=\{\vec w\in\mathbb{R}^n:\, \vec w=\sum_{i=1}^l\alpha_i\vec{v}_i,\,
     |\alpha_i|\leq q\}$ and vector $\vec{t}_\lambda\in \mathbb{R}^n$.
   \end{itemize}
\end{thm}
\begin{proof}
  Consider $y\in Y$ and  $Y_0=\mathrm{closure}(\mathcal{T}(y))\subseteq Y$. Clearly, $(Y_0,d)$ is compact and invariant under the action of $\mathcal{T}$. Let $F=\{\vec{v}_1,\ldots, \vec{v}_l\}$ and consider the $l$ commuting homeomorphisms of $Y_0$ given by $T_i=T_{-\vec{v}_i}$.

   By Lemma \ref{BrownTop}, for each $\epsilon'>0$ there exists $q'\in\mathbb{N}$ satisfying: for each $\lambda>0$  there exist $x_k\in Y_0$ and a collection $\{\vec{u}_i\}_{i\in \{1,\ldots, l\}}$, with $\vec{u}_i\in Q^l_{q'}$ and $\vec{u}_i=(u^i_1,\ldots, u^i_l)$, such that
  $$d(x_k,T_i^kT_1^{u^i_1}\ldots T_l^{u^i_l}(x_k))< \epsilon'/3$$
  for all $i\in\{1,\ldots,l\}$, where $k=\lceil\lambda\rceil$ is the smallest integer greater or equal than $\lambda$. Since $x_k\in Y_0$ is either a translation of $y$ or the limit of translations of $y$, by continuity we can find $\vec{v}_k\in\mathbb{R}^n$ such that $d(x_k,T_{\vec v_k}(y))<\epsilon'/3$ and
  $$|d(T_{\vec{v}_k}(y),T_i^kT_1^{u^i_1}\ldots T_l^{u^i_l}T_{\vec{v}_k}(y))-d(x_k,T_i^kT_1^{u^i_1}\ldots T_l^{u^i_l}(x_k))|<\epsilon'/3,$$
  hence
$$d(T_{\vec{v}_k}(y),T_i^kT_1^{u^i_1}\ldots T_l^{u^i_l}T_{\vec{v}_k}(y))< \epsilon',$$
  for all $i\in\{1,\ldots,l\}$.
  By the definition of the metric $d$, there exist $$z_{i,k}'\in T_i^kT_1^{u^i_1}\ldots T_l^{u^i_l}T_{\vec{v}_k}(y)[[B_{1/\epsilon'}]],\quad z_{i,k}''\in T_{\vec{v}_k}(y)[[B_{1/\epsilon'}]]\quad\textrm{and}\,\, h_{i,k}\in G[z_{i,k}'],$$ with $\theta(\Delta(z_{i,k}',\epsilon'),\|h_{i,k}\|_{G[z_{i,k}']})\leq \epsilon'$, such that $h_{i,k}(z_{i,k}')=z_{i,k}''$. Now consider $z_k''\subset T_{\vec v_k}(y)$ to be the connected component of $\bigcap _{i=1}^lz_{i,k}''$ whose support  contains the ball $B_{1/\epsilon'}$.
   By construction,   $$ T_l^{-u^i_l}\ldots T_1^{-u^i_1}T_i^{-k}T_{-\vec v_k}h_{i,k}^{-1}(z''_k)\subset y.$$

On the other hand, since, by Lemma \ref{BrownTop}, $d(x_1,x_k)< \epsilon'/3$, we have $$d(T_{\vec v_1}(y),T_{\vec v_k}(y))\leq d(T_{\vec v_1}(y),x_1)+d(x_1,x_k)+d(x_k,T_{\vec v_k}(y))<\epsilon'.$$ Hence, there exist $w'_{1,k}\in T_{\vec v_1}(y)[[B_{1/\epsilon'}]]$, $w'_k\in  T_{\vec v_k}(y) [[B_{1/\epsilon'}]]$ and $f_k\in G[w'_{1,k}]$  with $w'_k=f_k(w'_{1,k})$ and $\theta(\Delta(w'_{1,k},\epsilon'),\|f_{k}\|_{G[w'_{1,k}]}) )\leq\epsilon'$. Set
$$y'=\bigcap_kf^{-1}_{k}(w'_k\cap z''_k)\subset T_{\vec{v}_1}(y).$$
By Lemma \ref{i}, the support of $y'$ contains $B_{1/\epsilon'-\epsilon'}$. Moreover,
 $$ T_l^{-u^i_l}\ldots T_1^{-u^i_1}T_i^{-k}T_{-\vec v_k}h_{i,k}^{-1}f_k(y')\subset y$$
and 
\begin{align*}
\theta(\Delta(y',\epsilon'/(1-\epsilon'^2)),\|h^{-1}_{i,k}f_{k}\|_{G[y']}) )&\leq \theta(\Delta(w'_{1,k},\epsilon'),\|f_{k}\|_{G[w'_{1,k}]}))+\theta(\Delta(z_{i,k}',\epsilon'),\|h_{i,k}\|_{G[z_{i,k}']})\\& \leq 2\epsilon'.\end{align*}
Write $g_{\lambda,i}=h_{i,k}^{-1}f_k$.
 By the continuity of the left and right multiplication by translations, and by the properties of $\theta\in \Theta$,  we can  choose $\epsilon'<\epsilon$  such that $B_{1/\epsilon}\subseteq \mathrm{supp}(y')$ and  $\|g_{\lambda,i}\|_{G[y_\lambda]}<\epsilon.$ The result holds with $\vec{t}_\lambda=-\vec{v_k}$, $\vec{u}=-\vec{v}_1$, $q=q'+1$ and $\vec{w}_{\lambda,i}=\sum_{j=1}^l\alpha_j\vec{v}_j\in Q_q(F)$, with $\alpha_j=u^i_j\in\mathbb{Z}$ if $j\neq i$ and $\alpha_i=(\lceil\lambda\rceil-\lambda)+u_i^i$.
\end{proof}

\begin{rem}
  For almost periodic tilings, this property  holds automatically.
However, it survives even  in the lack of almost periodicity.
\end{rem}

Let us give an informal pictorial illustration of Theorem \ref{BT}. Suppose that we have a tilling $y$ and a finite subset $F$ of $\mathbb{R}^n$, for example the set represented in Figure 2.
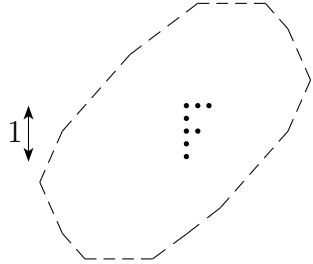
\begin{figure}[!htb]\label{Fsubis}
\psset{xunit=0.15cm,yunit=0.17cm,algebraic=true,dotstyle=o,dotsize=2pt 0,linewidth=0.4pt,arrowsize=3pt 2,arrowinset=0.25}
\begin{pspicture*}(-14.6,-10.13)(12.19,10.37)
\psline{<->}(-13,-2.4)(-13,2)
\psline[linewidth=.4pt,linestyle=dashed](2,10)(8,10)
\psline[linewidth=.4pt,linestyle=dashed](8,10)(10,8)
\psline[linewidth=.4pt,linestyle=dashed](10,8)(12,4)
\psline[linewidth=.4pt,linestyle=dashed](12,4)(10,0)
\psline[linewidth=.4pt,linestyle=dashed](10,0)(4,-6)
\psline[linewidth=.4pt,linestyle=dashed](4,-6)(0.97,-8.08)
\psline[linewidth=.4pt,linestyle=dashed](0.97,-8.08)(-2,-10)
\psline[linewidth=.4pt,linestyle=dashed](-2,-10)(-8,-10)
\psline[linewidth=.4pt,linestyle=dashed](-8,-10)(-10,-8)
\psline[linewidth=.4pt,linestyle=dashed](-10,-8)(-12,-4)
\psline[linewidth=.4pt,linestyle=dashed](-12,-4)(-10,0)
\psline[linewidth=.4pt,linestyle=dashed](-10,0)(-4,6)
\psline[linewidth=.4pt,linestyle=dashed](-4,6)(2,10)
\rput(-14.2,0){$1$}
\begin{scriptsize}
\psdots[dotstyle=*](1,0)
\psdots[dotstyle=*](1,1)
\psdots[dotstyle=*](1,2)
\psdots[dotstyle=*](1,-1)
\psdots[dotstyle=*](1,-2)
\psdots[dotstyle=*](3,2)
\psdots[dotstyle=*](2,0)
\psdots[dotstyle=*](2,2)
\end{scriptsize}
\end{pspicture*}
 \caption{The finite set $F$ and the convex hull of $Q_q(F)$.}
 \end{figure}

Given $\epsilon >0$, there exist $q\in\mathbb{N}$ and a patch $y'$, whose support contains the ball of radius $1/\epsilon$ about the origin, such that, for each scale factor $\lambda$, there exits a vector $\vec{t}_\lambda$ so that copies of  $y'$ appear in the tilling ``nearly"  centered (in the sense that, for each $\vec{v}_i\in F$, $\|g_{\lambda,i}\|_{G[y_\lambda]}<\epsilon$) on $\lambda F+\vec{t}_\lambda$ up to ``bounded perturbations" (in the sense that $q$, and consequently $Q_q(F)$, does not depend on $\lambda$ and $\vec{w}_{\lambda,i}\in Q_q(F)$), as represented in Figure 3.

\begin{figure}[!htb]\label{Fsubis}
\psset{xunit=0.02cm,yunit=0.02cm,algebraic=true,dotstyle=o,dotsize=2pt 0,linewidth=0.1pt,arrowsize=3pt 2,arrowinset=0.25}
\begin{pspicture*}(1186,405.76)(1337.97,632.77)
\rput(1192,517){$\lambda$}
\psline[linewidth=0.4pt]{<->}(1202,420)(1202,620)
\psline[linewidth=0.4pt,linestyle=dashed,dash=1pt 1pt](1277.44,628)(1283.44,628)
\psline[linewidth=0.4pt,linestyle=dashed,dash=1pt 1pt](1283.44,628)(1285.44,626)
\psline[linewidth=0.4pt,linestyle=dashed,dash=1pt 1pt](1285.44,626)(1287.44,622)
\psline[linewidth=0.4pt,linestyle=dashed,dash=1pt 1pt](1287.44,622)(1285.44,618)
\psline[linewidth=0.4pt,linestyle=dashed,dash=1pt 1pt](1285.44,618)(1279.44,612)
\psline[linewidth=0.4pt,linestyle=dashed,dash=1pt 1pt](1279.44,612)(1276.41,609.93)
\psline[linewidth=0.4pt,linestyle=dashed,dash=1pt 1pt](1276.41,609.93)(1273.44,608)
\psline[linewidth=0.4pt,linestyle=dashed,dash=1pt 1pt](1273.44,608)(1267.44,608)
\psline[linewidth=0.4pt,linestyle=dashed,dash=1pt 1pt](1267.44,608)(1265.44,610)
\psline[linewidth=0.4pt,linestyle=dashed,dash=1pt 1pt](1265.44,610)(1263.44,614)
\psline[linewidth=0.4pt,linestyle=dashed,dash=1pt 1pt](1263.44,614)(1265.44,618)
\psline[linewidth=0.4pt,linestyle=dashed,dash=1pt 1pt](1265.44,618)(1271.44,624)
\psline[linewidth=0.4pt,linestyle=dashed,dash=1pt 1pt](1271.44,624)(1277.44,628)
\psline[linewidth=0.4pt,linestyle=dashed,dash=1pt 1pt](1228.44,530)(1234.44,530)
\psline[linewidth=0.4pt,linestyle=dashed,dash=1pt 1pt](1234.44,530)(1236.44,528)
\psline[linewidth=0.4pt,linestyle=dashed,dash=1pt 1pt](1236.44,528)(1238.44,524)
\psline[linewidth=0.4pt,linestyle=dashed,dash=1pt 1pt](1238.44,524)(1236.44,520)
\psline[linewidth=0.4pt,linestyle=dashed,dash=1pt 1pt](1236.44,520)(1230.44,514)
\psline[linewidth=0.4pt,linestyle=dashed,dash=1pt 1pt](1230.44,514)(1227.41,511.93)
\psline[linewidth=0.4pt,linestyle=dashed,dash=1pt 1pt](1227.41,511.93)(1224.44,510)
\psline[linewidth=0.4pt,linestyle=dashed,dash=1pt 1pt](1224.44,510)(1218.44,510)
\psline[linewidth=0.4pt,linestyle=dashed,dash=1pt 1pt](1218.44,510)(1216.44,512)
\psline[linewidth=0.4pt,linestyle=dashed,dash=1pt 1pt](1216.44,512)(1214.44,516)
\psline[linewidth=0.4pt,linestyle=dashed,dash=1pt 1pt](1214.44,516)(1216.44,520)
\psline[linewidth=0.4pt,linestyle=dashed,dash=1pt 1pt](1216.44,520)(1222.44,526)
\psline[linewidth=0.4pt,linestyle=dashed,dash=1pt 1pt](1222.44,526)(1228.44,530)
\psline[linewidth=0.4pt,linestyle=dashed,dash=1pt 1pt](1228.44,579)(1234.44,579)
\psline[linewidth=0.4pt,linestyle=dashed,dash=1pt 1pt](1234.44,579)(1236.44,577)
\psline[linewidth=0.4pt,linestyle=dashed,dash=1pt 1pt](1236.44,577)(1238.44,573)
\psline[linewidth=0.4pt,linestyle=dashed,dash=1pt 1pt](1238.44,573)(1236.44,569)
\psline[linewidth=0.4pt,linestyle=dashed,dash=1pt 1pt](1236.44,569)(1230.44,563)
\psline[linewidth=0.4pt,linestyle=dashed,dash=1pt 1pt](1230.44,563)(1227.41,560.93)
\psline[linewidth=0.4pt,linestyle=dashed,dash=1pt 1pt](1227.41,560.93)(1224.44,559)
\psline[linewidth=0.4pt,linestyle=dashed,dash=1pt 1pt](1224.44,559)(1218.44,559)
\psline[linewidth=0.4pt,linestyle=dashed,dash=1pt 1pt](1218.44,559)(1216.44,561)
\psline[linewidth=0.4pt,linestyle=dashed,dash=1pt 1pt](1216.44,561)(1214.44,565)
\psline[linewidth=0.4pt,linestyle=dashed,dash=1pt 1pt](1214.44,565)(1216.44,569)
\psline[linewidth=0.4pt,linestyle=dashed,dash=1pt 1pt](1216.44,569)(1222.44,575)
\psline[linewidth=0.4pt,linestyle=dashed,dash=1pt 1pt](1222.44,575)(1228.44,579)
\psline[linewidth=0.4pt,linestyle=dashed,dash=1pt 1pt](1228.44,628)(1234.44,628)
\psline[linewidth=0.4pt,linestyle=dashed,dash=1pt 1pt](1234.44,628)(1236.44,626)
\psline[linewidth=0.4pt,linestyle=dashed,dash=1pt 1pt](1236.44,626)(1238.44,622)
\psline[linewidth=0.4pt,linestyle=dashed,dash=1pt 1pt](1238.44,622)(1236.44,618)
\psline[linewidth=0.4pt,linestyle=dashed,dash=1pt 1pt](1236.44,618)(1230.44,612)
\psline[linewidth=0.4pt,linestyle=dashed,dash=1pt 1pt](1230.44,612)(1227.41,609.93)
\psline[linewidth=0.4pt,linestyle=dashed,dash=1pt 1pt](1227.41,609.93)(1224.44,608)
\psline[linewidth=0.4pt,linestyle=dashed,dash=1pt 1pt](1224.44,608)(1218.44,608)
\psline[linewidth=0.4pt,linestyle=dashed,dash=1pt 1pt](1218.44,608)(1216.44,610)
\psline[linewidth=0.4pt,linestyle=dashed,dash=1pt 1pt](1216.44,610)(1214.44,614)
\psline[linewidth=0.4pt,linestyle=dashed,dash=1pt 1pt](1214.44,614)(1216.44,618)
\psline[linewidth=0.4pt,linestyle=dashed,dash=1pt 1pt](1216.44,618)(1222.44,624)
\psline[linewidth=0.4pt,linestyle=dashed,dash=1pt 1pt](1222.44,624)(1228.44,628)
\psline[linewidth=0.4pt,linestyle=dashed,dash=1pt 1pt](1228.44,481)(1234.44,481)
\psline[linewidth=0.4pt,linestyle=dashed,dash=1pt 1pt](1234.44,481)(1236.44,479)
\psline[linewidth=0.4pt,linestyle=dashed,dash=1pt 1pt](1236.44,479)(1238.44,475)
\psline[linewidth=0.4pt,linestyle=dashed,dash=1pt 1pt](1238.44,475)(1236.44,471)
\psline[linewidth=0.4pt,linestyle=dashed,dash=1pt 1pt](1236.44,471)(1230.44,465)
\psline[linewidth=0.4pt,linestyle=dashed,dash=1pt 1pt](1230.44,465)(1227.41,462.93)
\psline[linewidth=0.4pt,linestyle=dashed,dash=1pt 1pt](1227.41,462.93)(1224.44,461)
\psline[linewidth=0.4pt,linestyle=dashed,dash=1pt 1pt](1224.44,461)(1218.44,461)
\psline[linewidth=0.4pt,linestyle=dashed,dash=1pt 1pt](1218.44,461)(1216.44,463)
\psline[linewidth=0.4pt,linestyle=dashed,dash=1pt 1pt](1216.44,463)(1214.44,467)
\psline[linewidth=0.4pt,linestyle=dashed,dash=1pt 1pt](1214.44,467)(1216.44,471)
\psline[linewidth=0.4pt,linestyle=dashed,dash=1pt 1pt](1216.44,471)(1222.44,477)
\psline[linewidth=0.4pt,linestyle=dashed,dash=1pt 1pt](1222.44,477)(1228.44,481)
\psline[linewidth=0.4pt,linestyle=dashed,dash=1pt 1pt](1228.44,432)(1234.44,432)
\psline[linewidth=0.4pt,linestyle=dashed,dash=1pt 1pt](1234.44,432)(1236.44,430)
\psline[linewidth=0.4pt,linestyle=dashed,dash=1pt 1pt](1236.44,430)(1238.44,426)
\psline[linewidth=0.4pt,linestyle=dashed,dash=1pt 1pt](1238.44,426)(1236.44,422)
\psline[linewidth=0.4pt,linestyle=dashed,dash=1pt 1pt](1236.44,422)(1230.44,416)
\psline[linewidth=0.4pt,linestyle=dashed,dash=1pt 1pt](1230.44,416)(1227.41,413.93)
\psline[linewidth=0.4pt,linestyle=dashed,dash=1pt 1pt](1227.41,413.93)(1224.44,412)
\psline[linewidth=0.4pt,linestyle=dashed,dash=1pt 1pt](1224.44,412)(1218.44,412)
\psline[linewidth=0.4pt,linestyle=dashed,dash=1pt 1pt](1218.44,412)(1216.44,414)
\psline[linewidth=0.4pt,linestyle=dashed,dash=1pt 1pt](1216.44,414)(1214.44,418)
\psline[linewidth=0.4pt,linestyle=dashed,dash=1pt 1pt](1214.44,418)(1216.44,422)
\psline[linewidth=0.4pt,linestyle=dashed,dash=1pt 1pt](1216.44,422)(1222.44,428)
\psline[linewidth=0.4pt,linestyle=dashed,dash=1pt 1pt](1222.44,428)(1228.44,432)
\psline[linewidth=0.4pt,linestyle=dashed,dash=1pt 1pt](1277.44,530)(1283.44,530)
\psline[linewidth=0.4pt,linestyle=dashed,dash=1pt 1pt](1283.44,530)(1285.44,528)
\psline[linewidth=0.4pt,linestyle=dashed,dash=1pt 1pt](1285.44,528)(1287.44,524)
\psline[linewidth=0.4pt,linestyle=dashed,dash=1pt 1pt](1287.44,524)(1285.44,520)
\psline[linewidth=0.4pt,linestyle=dashed,dash=1pt 1pt](1285.44,520)(1279.44,514)
\psline[linewidth=0.4pt,linestyle=dashed,dash=1pt 1pt](1279.44,514)(1276.41,511.93)
\psline[linewidth=0.4pt,linestyle=dashed,dash=1pt 1pt](1276.41,511.93)(1273.44,510)
\psline[linewidth=0.4pt,linestyle=dashed,dash=1pt 1pt](1273.44,510)(1267.44,510)
\psline[linewidth=0.4pt,linestyle=dashed,dash=1pt 1pt](1267.44,510)(1265.44,512)
\psline[linewidth=0.4pt,linestyle=dashed,dash=1pt 1pt](1265.44,512)(1263.44,516)
\psline[linewidth=0.4pt,linestyle=dashed,dash=1pt 1pt](1263.44,516)(1265.44,520)
\psline[linewidth=0.4pt,linestyle=dashed,dash=1pt 1pt](1265.44,520)(1271.44,526)
\psline[linewidth=0.4pt,linestyle=dashed,dash=1pt 1pt](1271.44,526)(1277.44,530)
\psline[linewidth=0.4pt,linestyle=dashed,dash=1pt 1pt](1326.44,628)(1332.44,628)
\psline[linewidth=0.4pt,linestyle=dashed,dash=1pt 1pt](1332.44,628)(1334.44,626)
\psline[linewidth=0.4pt,linestyle=dashed,dash=1pt 1pt](1334.44,626)(1336.44,622)
\psline[linewidth=0.4pt,linestyle=dashed,dash=1pt 1pt](1336.44,622)(1334.44,618)
\psline[linewidth=0.4pt,linestyle=dashed,dash=1pt 1pt](1334.44,618)(1328.44,612)
\psline[linewidth=0.4pt,linestyle=dashed,dash=1pt 1pt](1328.44,612)(1325.41,609.93)
\psline[linewidth=0.4pt,linestyle=dashed,dash=1pt 1pt](1325.41,609.93)(1322.44,608)
\psline[linewidth=0.4pt,linestyle=dashed,dash=1pt 1pt](1322.44,608)(1316.44,608)
\psline[linewidth=0.4pt,linestyle=dashed,dash=1pt 1pt](1316.44,608)(1314.44,610)
\psline[linewidth=0.4pt,linestyle=dashed,dash=1pt 1pt](1314.44,610)(1312.44,614)
\psline[linewidth=0.4pt,linestyle=dashed,dash=1pt 1pt](1312.44,614)(1314.44,618)
\psline[linewidth=0.4pt,linestyle=dashed,dash=1pt 1pt](1314.44,618)(1320.44,624)
\psline[linewidth=0.4pt,linestyle=dashed,dash=1pt 1pt](1320.44,624)(1326.44,628)
\psline(1227.29,622.54)(1231.12,620.66)
\psline(1231.12,620.66)(1231.02,616.34)
\psline(1227.29,622.54)(1224.68,619.86)
\psline(1224.68,619.86)(1224.05,618.1)
\psline(1224.05,618.1)(1225.24,615.7)
\psline(1225.24,615.7)(1228.9,614.83)
\psline(1228.9,614.83)(1231.02,616.34)
\psline(1228.34,618.43)(1227.29,622.54)
\psline(1228.34,618.43)(1231.12,620.66)
\psline(1228.34,618.43)(1231.02,616.34)
\psline(1228.34,618.43)(1228.9,614.83)
\psline(1228.34,618.43)(1225.24,615.7)
\psline(1226.73,618.56)(1225.24,615.7)
\psline(1226.73,618.56)(1228.34,618.43)
\psline(1226.73,618.56)(1227.29,622.54)
\psline(1226.73,618.56)(1231.12,620.66)
\psline(1226.73,618.56)(1228.9,614.83)
\psline(1226.73,618.56)(1224.68,619.86)
\psline(1224.05,618.1)(1226.73,618.56)
\psline(1228.34,618.43)(1224.63,616.94)
\psline(1226.02,621.23)(1226.73,618.56)
\psline(1227.09,617.93)(1225.24,615.7)
\psline(1226.73,618.56)(1224.63,616.94)
\psline(1281.46,626)(1285.29,624.13)
\psline(1285.29,624.13)(1285.19,619.8)
\psline(1281.46,626)(1278.86,623.32)
\psline(1278.86,623.32)(1278.23,621.57)
\psline(1278.23,621.57)(1279.42,619.17)
\psline(1279.42,619.17)(1283.07,618.3)
\psline(1283.07,618.3)(1285.19,619.8)
\psline(1282.51,621.9)(1281.46,626)
\psline(1282.51,621.9)(1285.29,624.13)
\psline(1282.51,621.9)(1285.19,619.8)
\psline(1282.51,621.9)(1283.07,618.3)
\psline(1282.51,621.9)(1279.42,619.17)
\psline(1280.9,622.02)(1279.42,619.17)
\psline(1280.9,622.02)(1282.51,621.9)
\psline(1280.9,622.02)(1281.46,626)
\psline(1280.9,622.02)(1285.29,624.13)
\psline(1280.9,622.02)(1283.07,618.3)
\psline(1280.9,622.02)(1278.86,623.32)
\psline(1278.23,621.57)(1280.9,622.02)
\psline(1282.51,621.9)(1278.8,620.41)
\psline(1280.2,624.7)(1280.9,622.02)
\psline(1281.27,621.4)(1279.42,619.17)
\psline(1280.9,622.02)(1278.8,620.41)
\psline(1326.49,618.63)(1330.32,616.76)
\psline(1330.32,616.76)(1330.22,612.43)
\psline(1326.49,618.63)(1323.89,615.95)
\psline(1323.89,615.95)(1323.26,614.2)
\psline(1323.26,614.2)(1324.44,611.8)
\psline(1324.44,611.8)(1328.1,610.93)
\psline(1328.1,610.93)(1330.22,612.43)
\psline(1327.54,614.52)(1326.49,618.63)
\psline(1327.54,614.52)(1330.32,616.76)
\psline(1327.54,614.52)(1330.22,612.43)
\psline(1327.54,614.52)(1328.1,610.93)
\psline(1327.54,614.52)(1324.44,611.8)
\psline(1325.93,614.65)(1324.44,611.8)
\psline(1325.93,614.65)(1327.54,614.52)
\psline(1325.93,614.65)(1326.49,618.63)
\psline(1325.93,614.65)(1330.32,616.76)
\psline(1325.93,614.65)(1328.1,610.93)
\psline(1325.93,614.65)(1323.89,615.95)
\psline(1323.26,614.2)(1325.93,614.65)
\psline(1327.54,614.52)(1323.83,613.04)
\psline(1325.22,617.33)(1325.93,614.65)
\psline(1326.3,614.02)(1324.44,611.8)
\psline(1325.93,614.65)(1323.83,613.04)
\psline(1222.66,568.49)(1226.5,566.61)
\psline(1226.5,566.61)(1226.4,562.29)
\psline(1222.66,568.49)(1220.06,565.81)
\psline(1220.06,565.81)(1219.43,564.06)
\psline(1219.43,564.06)(1220.62,561.66)
\psline(1220.62,561.66)(1224.28,560.78)
\psline(1224.28,560.78)(1226.4,562.29)
\psline(1223.72,564.38)(1222.66,568.49)
\psline(1223.72,564.38)(1226.5,566.61)
\psline(1223.72,564.38)(1226.4,562.29)
\psline(1223.72,564.38)(1224.28,560.78)
\psline(1223.72,564.38)(1220.62,561.66)
\psline(1222.11,564.51)(1220.62,561.66)
\psline(1222.11,564.51)(1223.72,564.38)
\psline(1222.11,564.51)(1222.66,568.49)
\psline(1222.11,564.51)(1226.5,566.61)
\psline(1222.11,564.51)(1224.28,560.78)
\psline(1222.11,564.51)(1220.06,565.81)
\psline(1219.43,564.06)(1222.11,564.51)
\psline(1223.72,564.38)(1220.01,562.89)
\psline(1221.4,567.18)(1222.11,564.51)
\psline(1222.47,563.88)(1220.62,561.66)
\psline(1222.11,564.51)(1220.01,562.89)
\psline(1230.51,528.33)(1234.34,526.46)
\psline(1234.34,526.46)(1234.24,522.14)
\psline(1230.51,528.33)(1227.91,525.65)
\psline(1227.91,525.65)(1227.28,523.9)
\psline(1227.28,523.9)(1228.46,521.5)
\psline(1228.46,521.5)(1232.12,520.63)
\psline(1232.12,520.63)(1234.24,522.14)
\psline(1231.56,524.23)(1230.51,528.33)
\psline(1231.56,524.23)(1234.34,526.46)
\psline(1231.56,524.23)(1234.24,522.14)
\psline(1231.56,524.23)(1232.12,520.63)
\psline(1231.56,524.23)(1228.46,521.5)
\psline(1229.95,524.35)(1228.46,521.5)
\psline(1229.95,524.35)(1231.56,524.23)
\psline(1229.95,524.35)(1230.51,528.33)
\psline(1229.95,524.35)(1234.34,526.46)
\psline(1229.95,524.35)(1232.12,520.63)
\psline(1229.95,524.35)(1227.91,525.65)
\psline(1227.28,523.9)(1229.95,524.35)
\psline(1231.56,524.23)(1227.85,522.74)
\psline(1229.24,527.03)(1229.95,524.35)
\psline(1230.32,523.73)(1228.46,521.5)
\psline(1229.95,524.35)(1227.85,522.74)
\psline(1278,523.98)(1281.83,522.11)
\psline(1281.83,522.11)(1281.73,517.79)
\psline(1278,523.98)(1275.4,521.3)
\psline(1275.4,521.3)(1274.77,519.55)
\psline(1274.77,519.55)(1275.95,517.15)
\psline(1275.95,517.15)(1279.61,516.28)
\psline(1279.61,516.28)(1281.73,517.79)
\psline(1279.05,519.88)(1278,523.98)
\psline(1279.05,519.88)(1281.83,522.11)
\psline(1279.05,519.88)(1281.73,517.79)
\psline(1279.05,519.88)(1279.61,516.28)
\psline(1279.05,519.88)(1275.95,517.15)
\psline(1277.44,520)(1275.95,517.15)
\psline(1277.44,520)(1279.05,519.88)
\psline(1277.44,520)(1278,523.98)
\psline(1277.44,520)(1281.83,522.11)
\psline(1277.44,520)(1279.61,516.28)
\psline(1277.44,520)(1275.4,521.3)
\psline(1274.77,519.55)(1277.44,520)
\psline(1279.05,519.88)(1275.34,518.39)
\psline(1276.73,522.68)(1277.44,520)
\psline(1277.8,519.38)(1275.95,517.15)
\psline(1277.44,520)(1275.34,518.39)
\psline(1228.71,478.6)(1232.54,476.73)
\psline(1232.54,476.73)(1232.45,472.41)
\psline(1228.71,478.6)(1226.11,475.93)
\psline(1226.11,475.93)(1225.48,474.17)
\psline(1225.48,474.17)(1226.67,471.77)
\psline(1226.67,471.77)(1230.32,470.9)
\psline(1230.32,470.9)(1232.45,472.41)
\psline(1229.76,474.5)(1228.71,478.6)
\psline(1229.76,474.5)(1232.54,476.73)
\psline(1229.76,474.5)(1232.45,472.41)
\psline(1229.76,474.5)(1230.32,470.9)
\psline(1229.76,474.5)(1226.67,471.77)
\psline(1228.15,474.63)(1226.67,471.77)
\psline(1228.15,474.63)(1229.76,474.5)
\psline(1228.15,474.63)(1228.71,478.6)
\psline(1228.15,474.63)(1232.54,476.73)
\psline(1228.15,474.63)(1230.32,470.9)
\psline(1228.15,474.63)(1226.11,475.93)
\psline(1225.48,474.17)(1228.15,474.63)
\psline(1229.76,474.5)(1226.05,473.01)
\psline(1227.45,477.3)(1228.15,474.63)
\psline(1228.52,474)(1226.67,471.77)
\psline(1228.15,474.63)(1226.05,473.01)
\psline(1227.97,417.91)(1231.8,416.03)
\psline(1231.8,416.03)(1231.7,411.71)
\psline(1227.97,417.91)(1225.37,415.23)
\psline(1225.37,415.23)(1224.74,413.48)
\psline(1224.74,413.48)(1225.92,411.08)
\psline(1225.92,411.08)(1229.58,410.2)
\psline(1229.58,410.2)(1231.7,411.71)
\psline(1229.02,413.8)(1227.97,417.91)
\psline(1229.02,413.8)(1231.8,416.03)
\psline(1229.02,413.8)(1231.7,411.71)
\psline(1229.02,413.8)(1229.58,410.2)
\psline(1229.02,413.8)(1225.92,411.08)
\psline(1227.41,413.93)(1225.92,411.08)
\psline(1227.41,413.93)(1229.02,413.8)
\psline(1227.41,413.93)(1227.97,417.91)
\psline(1227.41,413.93)(1231.8,416.03)
\psline(1227.41,413.93)(1229.58,410.2)
\psline(1227.41,413.93)(1225.37,415.23)
\psline(1224.74,413.48)(1227.41,413.93)
\psline(1229.02,413.8)(1225.31,412.31)
\psline(1226.7,416.6)(1227.41,413.93)
\psline(1227.78,413.3)(1225.92,411.08)
\psline(1227.41,413.93)(1225.31,412.31)
\end{pspicture*}
 \caption{Copies of $y'$ appear  in $y$ ``nearly"  centered  on $\lambda F+\vec{t}_\lambda$ up to ``bounded perturbations" in the structure of the homothetic copies of $F$.}
 \end{figure}
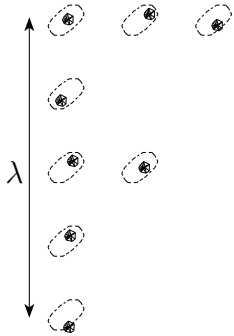

\end{document}